 \documentclass[final,3p,times]{elsarticle}

\usepackage{amssymb}
\usepackage{amsmath}
\usepackage{bm}
\usepackage{mathrsfs}
\usepackage{subfigure}
\usepackage{algorithm}%
\usepackage{algorithmicx}%
\usepackage{algpseudocode}%
\usepackage{booktabs}
\usepackage{hyperref}
\usepackage{amsthm}

\newtheorem{theorem}{Theorem}
\newtheorem{lemma}{Lemma}
 \journal{Computer Methods in Applied Mechanics and Engineering}

\begin{document}

\begin{frontmatter}



\title{An adaptive perfectly matched layer finite element method for acoustic-elastic interaction in periodic structures}

\author[1]{Sijia Li}
\author[1,2]{Lei Lin}
\author[1]{Junliang Lv\corref{cor1}}
\ead{lvjl@jlu.edu.cn}

\cortext[cor1]{Corresponding author}

\affiliation[1]{%
	organization={School of Mathematics, Jilin University},
	addressline={Qianjin Street},
	city={Changchun},
	postcode={130012},
	state={Jilin Province},
	country={China}
}

\affiliation[2]{%
	organization={Institute of Computational Mathematics and Scientific/Engineering Computing},
	city={Beijing},
	postcode={100190},
	country={China}
}

\begin{abstract}
This paper considers the scattering of a time-harmonic acoustic plane wave by an elastic body with an unbounded periodic surface. The original problem can be confined to the analysis of the fields in one periodic cell. With the help of the perfectly matched layer (PML) technique, we can truncate the unbounded physical domain into a bounded computational domain. By respectively constructing the equivalent transparent boundary conditions of acoustic and elastic waves simultaneously, the well-posedness and exponential convergence of the solution to the associated truncated PML problem are established. The finite element method is applied to solve the PML problem of acoustic-elastic interaction. To address the singularity caused by the non-smooth surface of the elastic body, we establish a residual-type a posteriori error estimate and develop an adaptive PML finite element algorithm. Several numerical examples are presented to demonstrate the effectiveness of the proposed adaptive algorithm.
\end{abstract}

%


\end{frontmatter}



\section{Introduction}
\label{sec1}
Direct scattering \cite{Bramble107, Hohage03, Jiang22, Kirsch25, LiSi25, Lin26, Lin24} and inverse scattering \cite{Gao25, Huang21, Huang25, LLN25, LLW25} have long been an active area of research in applied mathematics and computational science due to their broad applications in science and engineering. In particular, the wave scattering in periodic structures \cite{Bao21,Niu24,Wang25}, also known as diffraction gratings, plays an essential role in the design and analysis of optical components, such as increasing the efficiency of beam splitters, improving the sensitivity of optical sensors,  and enhancing the performance of antireflective coatings \cite{Bao01, Dobson93}. This paper is concerned with the scattering of the time-harmonic acoustic plane wave by an elastic body with an unbounded periodic surface. The region above the surface is filled with a homogeneous compressible inviscid fluid medium, while the region below is occupied by an isotropic and linearly elastic solid material. Due to the external incident acoustic field, an elastic wave is excited inside the solid, and the incident acoustic wave is scattered back into the fluid. Such a phenomenon is commonly known as the acoustic-elastic interaction in periodic structures \cite{Claeys83, Hu16, Lin25,WangL25}, which has received wide attention due to its significant applications in underwater acoustics and ultrasonic nondestructive evaluation \cite{Akkose, Declercq05}. 

As in many other diffraction problems, the governing equations of acoustic-elastic interaction are formulated in unbounded domains, and the scattered and transmitted fields must satisfy appropriate radiation conditions at infinity \cite{Bao22}. For numerical simulation, it is essential to truncate the unbounded regions above and below the periodic interface and impose appropriate boundary conditions on artificial boundaries. For this purpose, several effective truncation methods have been proposed, such as the perfectly matched layer (PML) technique \cite{Berenger94}, the absorbing boundary conditions (ABCs) \cite{Engquist77}, and the transparent boundary conditions (TBCs) \cite{Grote04, Jiang17}. In this paper, we introduce the PML technique for domain truncation. The fundamental principle of the PML method is to design a finite-thickness absorbing layer with suitably chosen complex coordinate stretching to surround the region of interest. Due to its ease of implementation and high computational efficiency, the PML method has been widely adopted for solving various wave scattering problems, such as acoustic waves \cite{LiY25, Lu23, Lu25}, elastic waves \cite{Bramble10, Chen16, Jiang18}, electromagnetic waves \cite{Bao24,Bao05}, and biharmonic waves \cite{LiIMA25}. However, due to the complexity of multiphysics coupling, a rigorous theoretical analysis of the PML method for the acoustic-elastic interaction in periodic structures is still lacking. 

In practical applications, the grating profiles are often only piecewise smooth and may contain reentrant corners. Such geometric singularities usually lead to reduced regularity of the acoustic and elastic fields near the non-smooth surface, which will severely degrade the computational efficiency of uniform mesh refinement. To address this difficulty, adaptive finite element methods (AFEMs)  provide an efficient numerical strategy, since they are able to concentrate computational effort only in regions where the solution exhibits singular behavior \cite{Babuska78}. Combined with the PML technique, an efficient adaptive finite element method was firstly developed for solving the diffraction grating problem \cite{Chen03}. It was shown that the a posteriori error estimate consists of the finite element discretization error and the PML truncation error, where the latter decays exponentially with respect to the PML parameters, such as the thickness of the layer and the medium properties. Due to its superior numerical behavior, the adaptive finite element PML method was quickly developed to solve a variety of scattering problems, such as obstacle scattering \cite{ Chen08, Chen05, JiangLi17}, periodic diffraction \cite{Bao10, Jiang17-3, Zhou18}, and open cavity problems \cite{ChenY21}. Nevertheless, most of these results have been largely limited to single-physics scattering problems. For the problem of acoustic-elastic interaction in periodic structures, the rigorous a posteriori analysis remains underdeveloped.

Motivated by these considerations, the current paper develops and analyzes an adaptive perfectly matched layer finite element method for the acoustic-elastic interaction in periodic structures.  We firstly derive a truncated PML formulation of acoustic-elastic interaction and establish its well-posedness. Moreover, we provide a rigorous convergence analysis showing that the PML approximation error decays exponentially with respect to the acoustic and elastic PML parameters simultaneously. Based on the truncated PML formulation, we then construct a residual-type a posteriori error estimator that explicitly incorporates the PML truncation error as well as the finite element discretization error. Utilizing the a posteriori error estimate, we propose an adaptive PML-FEM algorithm. Numerical experiments are also presented to demonstrate the effectiveness of the proposed adaptive algorithm, in particular for configurations with geometric singularities.

The remainder of this paper is organized as follows. In Section \ref{SECTION2}, we introduce the mathematical formulation of the acoustic-elastic interaction problem in periodic structures and the associated weak formulation. The truncated PML model and the convergence analysis of the PML approximation are presented in Section \ref{SECTION3}. Section \ref{SECTION4} is devoted to the finite element discretization and the residual-based a posteriori error estimate. In Section \ref{SECTION5}, we describe the implementation of the adaptive algorithm and show numerical results. Conclusions and future extensions are given in Section \ref{SECTION6}.

\section{Problem Formulation}\label{SECTION2}
In this section, we introduce the original model and its corresponding weak formulation of the acoustic-elastic interaction problem in periodic structures.
\subsection{Original model}
Assume that the surface of the elastic solid is periodic in the $x_1$-axis with the period $\Lambda$. Due to the periodic structure of the elastic body, the problem
can be restricted into a single periodic cell, where 
$x_1 \in(0, \Lambda)$.
Denote the surface of the elastic body by
$$
\Gamma=\left\{\boldsymbol{x}=\left(x_1, x_2\right) \in \mathbb{R}^2: x_1 \in (0, \Lambda), x_2=f\left(x_1\right) \right\},
$$
where $f$ is a Lipschitz function with period $\Lambda > 0$. Denote by 
\begin{align*}
	\Omega_\mathrm{u}=\left\{\boldsymbol{x} \in \mathbb{R}^2: x_1 \in (0, \Lambda), x_2> f\left(x_1\right)\right\}
\end{align*}
the domain above the periodic surface, which is filled with a homogeneous compressible inviscid fluid with a constant mass density $\rho_f>0$. The domain below the surface is denoted by
\begin{align*}
	\Omega_\mathrm{d}=\left\{\boldsymbol{x} \in \mathbb{R}^2: x_1 \in (0, \Lambda), x_2<f\left(x_1\right)\right\},
\end{align*}
which is occupied by an isotropic homogeneous elastic solid body characterized by the mass density  $\rho>0$ and the Lam\text{\'e} constants $\lambda, \mu \in \mathbb{R}$ satisfying $\mu>0$ and $\lambda+\mu>0$. Define the artificial boundaries above and below the fluid-solid interface as $\Gamma_+=\left\{\boldsymbol{x} \in \mathbb{R}^2: x_1 \in(0, \Lambda), x_2=h_1 \right\}$ and $\Gamma_-=\left\{\boldsymbol{x} \in \mathbb{R}^2: x_1 \in(0, \Lambda), x_2=h_2 \right\}$, where $h_1$ and $h_2$ are constants. Let
\begin{align*}
	\Omega_+ & = \left\{\boldsymbol{x} \in \mathbb{R}^2: x_1 \in \left(0, \Lambda\right), f\left(x_1\right)<x_2<h_1\right\}, \\
	\Omega_- & = \left\{\boldsymbol{x} \in \mathbb{R}^2:  x_1 \in \left(0, \Lambda\right), h_2<x_2<f\left(x_1\right)\right\}.
\end{align*}
Similarly, we define
\begin{align*}
	\Omega_+^\mathrm{e} &= \left\{\boldsymbol{x} \in \mathbb{R}^2: x_1 \in(0, \Lambda),  x_2>h_1 \right\}, \\
	\Omega_-^\mathrm{e} &= \left\{\boldsymbol{x} \in \mathbb{R}^2: x_1 \in(0, \Lambda), x_2<h_2 \right\}.
\end{align*}
Fig. \ref{model1} shows the problem geometry of acoustic-elastic interaction in periodic structures. Let $\boldsymbol{n}=\left(n_1, n_2\right)^{\top}$ be the unit normal vector to $\Gamma$  directed
into $\Omega_\mathrm{u}$, and denote 
$\Omega =\Omega_+ \cup \Omega_-.$
\begin{figure}[htbp]
	\centerline{\includegraphics[width=0.4\textwidth]{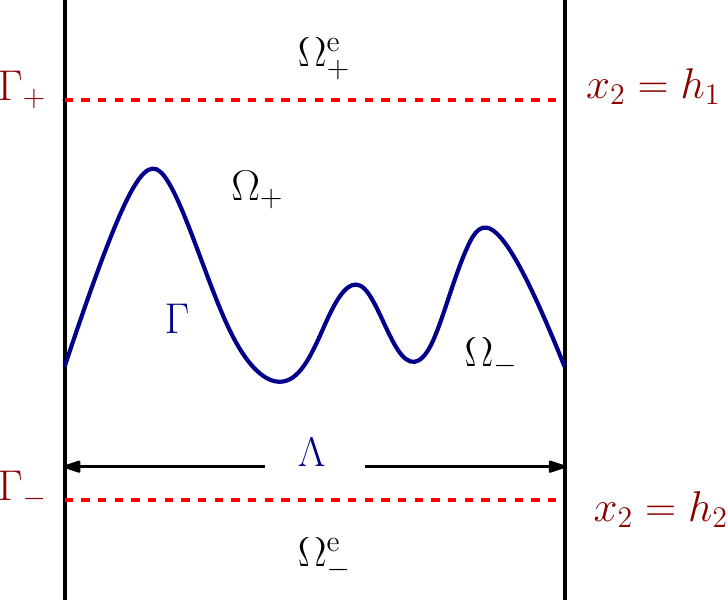}}
	\caption{Problem geometry of acoustic-elastic interaction in periodic structures.\label{model1}}
\end{figure}

Let $ p^\mathrm{in}=e^{\mathrm{i}\left(\alpha x_1-\beta x_2\right)}$ be the incoming acoustic wave, where $\alpha=\kappa \sin \theta, \beta=\kappa \cos \theta$, and $\theta \in(-\pi / 2, \pi / 2)$ is the incident angle. When the plane wave incident onto the elastic surface, the scattered acoustic wave $p^\mathrm{sc}$ in the fluid domain and the transmitted elastic wave $\boldsymbol{u}=\left(u_1, u_2\right)^{\top}$ in the solid domain are excited at the same time, where $p^\mathrm{sc}$ satisfies the Helmholtz equation
\begin{align*}
	\Delta p^\mathrm{sc}+\kappa^2 p^\mathrm{sc}=0 \quad &\text { in } \Omega_\mathrm{u},
\end{align*}
while $\boldsymbol{u}$ admits the Naiver equation
\begin{align}\label{navier}
	\Delta^* \boldsymbol{u} + \omega^2 \rho \boldsymbol{u}=0 \quad \text { in } \Omega_\mathrm{d}.
\end{align}
Here, $\kappa$ is the wavenumber, $\omega>0$ is the angular frequency, and $\Delta^* :=\mu \Delta + (\lambda+\mu) \nabla \nabla \cdot$. To ensure the continuity of the normal component of the velocity, the kinematic interface condition
\begin{align*}
	\partial_{\boldsymbol{n}} (p^\mathrm{in}+p^\mathrm{sc}) =\rho_f \omega^2 \boldsymbol{u} \cdot \boldsymbol{n} \quad \text { on } \Gamma
\end{align*}
is required. In addition, the following dynamic interface condition
\begin{align*}
	-(p^\mathrm{in}+p^\mathrm{sc}) \boldsymbol{n} =\boldsymbol{T} \boldsymbol{u} \quad \text { on } \Gamma
\end{align*}
is imposed to ensure the continuity of the traction (cf. \cite{Lin23, Xu21, Yin17}). Here, the traction of $\boldsymbol{u}$ is defined by 
\begin{align}\label{utrac}
	\boldsymbol{T}\boldsymbol{u}:=2 \mu \partial_{\boldsymbol{\nu}} \boldsymbol{u}+\lambda \boldsymbol{\nu} \nabla \cdot \boldsymbol{u}
	-\mu \left[\begin{array}{l}\left(\frac{\partial u_1}{\partial x_2}-\frac{\partial u_2}{\partial x_1}\right)\nu_2 \\ \left(\frac{\partial u_2}{\partial x_1}-\frac{\partial u_1}{\partial x_2}\right)\nu_1\end{array}\right],
\end{align}
where $\boldsymbol{\nu}=\left(\nu_1, \nu_2\right)^{\top}$ is the unit normal vector of $\Omega_-$. We want to reduce the original problem equivalently to a boundary value problem in $\Omega$. To this end, we introduce exact transparent boundary conditions for acoustic and elastic waves defined on $\Gamma_+$ and $\Gamma_-$, respectively.

Let $H^1(\Omega_+)$ and $H^1(\Omega_-)$ be the standard Sobolev spaces. The corresponding  quasi-periodic functional spaces are defined by
\begin{align*}
	H_{\mathrm{qp}}^1\left(\Omega_+\right):&=\left\{p \in H^1\left(\Omega_+\right): p\left(\Lambda, x_2\right)=p\left(0, x_2\right) e^{\mathrm{i} \alpha \Lambda}\right\},\\
	H_{\mathrm{qp}}^1\left(\Omega_-\right):&=\left\{u \in H^1\left(\Omega_-\right): u\left(\Lambda, x_2\right)=u\left(0, x_2\right) e^{\mathrm{i} \alpha \Lambda}\right\}.
\end{align*}
Given any quasi-periodic function $p \in H_{\mathrm{qp}}^1\left(\Omega_+\right)$, it admits a Fourier series expansion 
\begin{align*}
	p\left(x_1, x_2\right)=\sum_{n \in \mathbb{Z}} p_n(x_2) e^{\mathrm{i} \alpha_n x_1}, \quad p_n(x_2)=\frac{1}{\Lambda} \int_0^{\Lambda} p\left(x_1, x_2\right) e^{-\mathrm{i} \alpha_n x_1} \mathrm{d} x_1, 
\end{align*}
where $\alpha_n=2 n \pi / \Lambda + \alpha$. For any $s \in \mathbb{R}$, the trace functional space $H^s\left(\Gamma_+\right)$ is defined by
\begin{align*}
	H^s\left(\Gamma_+\right)=\left\{p \in L^2\left(\Gamma_+\right):\|p\|_{H^s\left(\Gamma_+\right)}<\infty\right\}
\end{align*}
with the trace norm 
\begin{align*}
	\|p\|_{H^s\left(\Gamma_+\right)}^2=\Lambda \sum_{n \in \mathbb{Z}}\left(1+\alpha_n^2\right)^s\left|p_n(h_1)\right|^2.
\end{align*}
Similarly, we can define the  trace functional space $H^s\left(\Gamma_-\right)$ and the corresponding trace norm $\|\cdot\|_{H^s\left(\Gamma_-\right)}$.
We firstly introduce the TBC of the scattered field $p^\mathrm{sc}$. Note that the Rayleigh expansion (cf. \cite{Wang15}) of $p^\mathrm{sc}$ can be defined by 
\begin{align}\label{ps rexp}
	p^\mathrm{sc}\left(x_1, x_2\right)=\sum_{n \in \mathbb{Z}} p_n^\mathrm{sc}\left(h_1\right) e^{\mathrm{i}\left(\alpha_n x_1+\beta_n\left(x_2-h_1\right)\right)}, \quad x_2>h_1,
\end{align}
where 
\begin{align*}
	\beta_n= \begin{cases}\left(\kappa^2-\alpha_n^2\right)^{1 / 2}, & \left|\alpha_n\right|<\kappa, \\ \mathrm{i}\left(\alpha_n^2-\kappa^2\right)^{1 / 2}, & \left|\alpha_n\right|>\kappa. \end{cases}
\end{align*}
We always assume that $\beta_n \neq 0$, i.e. $\left|\alpha_n\right| \neq \kappa$ to exclude Wood’s anomalies \cite{Bao22}.  Taking the normal derivative of (\ref{ps rexp}) with respect to $x_2$ and then evaluating it at $x_2=h_1$, one can get
\begin{equation*}
	\partial_{x_2} p^\mathrm{sc}(x_1,h_1)=\sum_{n \in \mathbb{Z}} \mathrm{i} \beta_n p^s_n(h_1) e^{\mathrm{i} \alpha_n {x_1}}.
\end{equation*}
It was shown in \cite{Wang15} that for any quasi-periodic function $p$, the Dirichlet-to-Neumann (DtN) operator on $\Gamma_+$ is defined by
\begin{align}\label{dw}
	\mathscr{T}_+ p\left(x_1, h_1\right):=\sum_{n \in \mathbb{Z}} \mathrm{i} \beta_n p_n(h_1) e^{\mathrm{i} \alpha_n x_1}.
\end{align}
Based on (\ref{dw}), the TBC for the scattered field $p^\mathrm{sc}$ can be defined by
\begin{align}\label{pdtn}
	\partial_{x_2} p^\mathrm{sc}=\mathscr{T}_+ p^\mathrm{sc} \quad \text { on }  \Gamma_+.
\end{align}
Next, we turn to derive the TBC for the transmitted field $\boldsymbol{u}$. Let the Jacobian matrix of $\boldsymbol{u}$ be
\begin{align*}
	\nabla \boldsymbol{u}=\left[\begin{array}{ll}
		\partial_{x_1} u_1 & \partial_{x_2} u_1 \\
		\partial_{x_1} u_2 & \partial_{x_2} u_2
	\end{array}\right].
\end{align*}
The operators $\operatorname{curl}$ and $\textbf{curl}$ are defined respectively by
\begin{align*}
	\operatorname{curl} \boldsymbol{u}=\partial_{x_1} u_2-\partial_{x_2} u_1, \quad \textbf{curl} u=\left[\partial_{x_2} u,-\partial_{x_1} u\right]^{\top}.
\end{align*}
The transmitted field $\bm{u}$ in $\Omega_-$ can be decomposed into the compressional  part $\bm{u}_\mathrm{p}$ and the shear part $\bm{u}_\mathrm{s}$, i.e.,
\begin{equation*}
	\bm{u}_\mathrm{p}=-\frac{1}{\kappa_{1}^2} \nabla \nabla \cdot \bm{u}, \quad \bm{u}_\mathrm{s}=\frac{1}{\kappa_{2}^2}\textbf{curl}\operatorname{curl} \bm{u},
\end{equation*}
where
\begin{align*}
	\kappa_1=\omega \sqrt{\frac{\rho}{2 \mu+\lambda}}, \quad \kappa_2=\omega \sqrt{\frac{\rho}{\mu}}
\end{align*}
are called the compressional and shear wavenumbers satisfying $\kappa_2>\kappa_1$. For any solution $\boldsymbol{u}$ of (\ref{navier}), it satisfies the Helmholtz decomposition
\begin{align} \label{hel decom}
	\boldsymbol{u}=\nabla \phi_1+\textbf{curl} \phi_2,
\end{align}
where $\phi_j$ ($j=1,2$) is called the scalar potential function.
Combining (\ref{navier}) and (\ref{hel decom}) yields 
\begin{align*}
	\Delta \phi_j+\kappa_j^2 \phi_j=0.
\end{align*}
Due to  the quasi-periodicity of $\boldsymbol{u}$, it follows from (\ref{hel decom}) that $\phi_j$ is also a quasi-periodic function in the $x_1$ direction with the period $\Lambda$ and admits the following Rayleigh expansion
\begin{align*}
	\phi_j\left(x_1, x_2\right)=\sum_{n \in \mathbb{Z}} \phi_n^{(j)}(h_2) e^{\mathrm{i}\left(\alpha_n x_1-\beta_n^{(j)}\left(x_2-h_2\right)\right)}, \quad x_2<h_2,
\end{align*}
where
\begin{align*}
	\beta_n^{(j)}= \begin{cases}\left(\kappa_j^2-\alpha_n^2\right)^{1 / 2}, & \left|\alpha_n\right|<\kappa_j, \\ \mathrm{i}\left(\alpha_n^2-\kappa_j^2\right)^{1 / 2}, & \left|\alpha_n\right|>\kappa_j.\end{cases}
\end{align*}
Similarly, we assume that $\kappa_j \neq\left|\alpha_n\right|$. 
It has been shown in \cite{Lin25} that the transmitted field $\boldsymbol{u}$ has the expansion
\begin{align}\label{urexp}
	\boldsymbol{u}\left(x_1, x_2\right)= & \sum_{n \in \mathbb{Z}} \frac{1}{\mathcal{X}_{n}}\left[\begin{array}{cc}
		\alpha_n^2 & -\alpha_n \beta_n^{(2)} \\
		-\alpha_n \beta_n^{(1)} & \beta_n^{(1)} \beta_n^{(2)}
	\end{array}\right] \boldsymbol{u}_n(h_2) \mathrm{e}^{\mathrm{i}\left(\alpha_n x_1-\beta_n^{(1)}\left(x_2-h_2\right)\right)} \notag\\
	& \quad +\frac{1}{\mathcal{X}_{n}}\left[\begin{array}{cc}
		\beta_n^{(1)} \beta_n^{(2)} & \alpha_n \beta_n^{(2)} \\
		\alpha_n \beta_n^{(1)} & \alpha_n^2
	\end{array}\right] \boldsymbol{u}_n(h_2) \mathrm{e}^{\mathrm{i}\left(\alpha_n x_1-\beta_n^{(2)}\left(x_2-h_2\right)\right)}, \quad x_2<h_2,
\end{align}
where $\mathcal{X}_{n} = \alpha_n^2 + \beta_n^{(1)}\beta_n^{(2)}$. 
Based on (\ref{utrac}), we can define the boundary operator on $\Gamma_-$ as
\begin{align}\label{uboun oper}
	\mathscr{T}_- \boldsymbol{u}:&= -2 \mu \partial_{x_2}\boldsymbol{u}+\lambda [0,-1]^{\top}\nabla \cdot \boldsymbol{u}+\mu \left[\begin{array}{c}\frac{\partial u_1}{\partial x_2}-\frac{\partial u_2}{\partial x_1} \\ 0\end{array}\right] \notag\\
	&=\left[\begin{array}{c}-\mu\left(\frac{\partial u_1}{\partial x_2}+\frac{\partial u_2}{\partial x_1}\right) \\ -(2 \mu+\lambda) \frac{\partial u_2}{\partial x_2}-\lambda \frac{\partial u_1}{\partial x_1}\end{array}\right].
\end{align}
From (\ref{urexp}) and (\ref{uboun oper}),  the TBC of the transmitted field $\boldsymbol{u}$ is defined by
\begin{align}\label{udtn}
	\mathscr{B} \boldsymbol{u}=\mathscr{T}_- \boldsymbol{u}=\sum_{n \in \mathbb{Z}} W_n\left[u_n^{(1)}(h_2), u_n^{(2)}(h_2)\right]^{\top} e^{\mathrm{i} \alpha_n x_1},
\end{align}
where $\mathscr{T}_-$ is the elastic DtN operator and the coefficient matrix is defined by
\begin{align*}
	W_n=\frac{\mathrm{i}}{\mathcal{X}_n}\left[\begin{array}{cc}
		\omega^2 \rho \beta_n^{(1)} & -2 \mu \alpha_n \mathcal{X}_n+\omega^2 \rho \alpha_n \\
		2 \mu \alpha_n \mathcal{X}_n-\omega^2 \rho \alpha_n & \omega^2 \rho \beta_n^{(2)}
	\end{array}\right].
\end{align*}
Then by the above transparent boundary conditions, one can obtain an acoustic-elastic interaction  boundary value problem in periodic structures: Given $p^\mathrm{in}$, seek quasi-periodic solutions $p^\mathrm{sc}$ and $\boldsymbol{u}$ such that
\begin{align}\label{DtNproblem}
	\Delta p^\mathrm{sc}+\kappa^2 p^\mathrm{sc}=0 & \quad\text { in } \Omega_+, \notag\\
	\Delta^* \boldsymbol{u} + \omega^2\rho \boldsymbol{u}=0 & \quad\text { in } \Omega_-, \notag \\
	\partial_{\boldsymbol{n}} (p^\mathrm{in}+p^\mathrm{sc})=\rho_f \omega^2 \boldsymbol{u} \cdot \boldsymbol{n} & \quad\text { on } \Gamma, \notag \\
	-(p^\mathrm{in}+p^\mathrm{sc}) \boldsymbol{n}=\boldsymbol{T} \boldsymbol{u} & \quad\text { on } \Gamma, \\
	\partial_{x_2} p^\mathrm{sc}=\mathscr{T}_+ p^\mathrm{sc} & \quad\text { on } \Gamma_+, \notag\\
	\mathscr{B} \boldsymbol{u}=\mathscr{T}_- \boldsymbol{u} &\quad \text { on } \Gamma_-.\notag
\end{align}
As shown in \cite{Hu16}, we know that the uniqueness of the solution for a bounded elastic body does not hold at Jones frequencies. Therefore, throughout the paper, we assume that the frequency $\omega$ is not a Jones frequency.

\subsection{Variational problem}
In this subsection, we give the variational formulation of the acoustic-elastic interaction in periodic structures. To facilitate subsequent theoretical analysis, we introduce the following product Sobolev space
\begin{align*}
	\mathscr{H}^1_{\mathrm{qp}}(\Omega):= H_{\mathrm{qp}}^1\left(\Omega_+\right) \times H_{\mathrm{qp}}^1\left(\Omega_-\right)^2=\left\{\boldsymbol{U}=(p, \boldsymbol{u}): p \in H_{\mathrm{qp}}^1\left(\Omega_+\right), \boldsymbol{u} \in H_{\mathrm{qp}}^1\left(\Omega_-\right)^2\right\}.
\end{align*}
The corresponding norm is defined by
\begin{align*}
	\|\boldsymbol{V}\|_{\mathscr{H}^1_\mathrm{qp}(\Omega)}^2 = & (\boldsymbol{V}, \boldsymbol{V})_{\mathscr{H}^1_{\mathrm{qp}}(\Omega)}\\
	= &  \int_{\Omega_+}(\nabla \varphi \cdot \nabla \overline{\varphi} + \varphi \overline{\varphi}) \mathrm{~d} \boldsymbol{x} + \int_{\Omega_-}\left(\mathcal{E}_{\lambda, \mu}(\boldsymbol{\psi}, \boldsymbol{\psi}) + \boldsymbol{\psi} \cdot \overline{\boldsymbol{\psi}}\right) \mathrm{d} \boldsymbol{x}, \quad \forall~\boldsymbol{V} = \left(\varphi, \boldsymbol{\psi}\right) \notag,
\end{align*}
where 
\begin{align*}
	\mathcal{E}_{\lambda, \mu}(\boldsymbol{u}, \boldsymbol{\psi}) 
	& = \lambda(\nabla \cdot \boldsymbol{u})(\nabla \cdot \overline{\boldsymbol{\psi}})+\frac{\mu}{2}\left(\nabla \boldsymbol{u}+\nabla \boldsymbol{u}^{\top}\right):\left(\nabla \overline{\boldsymbol{\psi}}+\nabla \overline{\boldsymbol{\psi}}^{\top}\right)\\
	&=  (2 \mu + \lambda)\left(\frac{\partial u_1}{\partial x_1} \frac{\partial \overline{\psi}_1}{\partial x_1}+\frac{\partial u_2}{\partial x_2} \frac{\partial \overline{\psi}_2}{\partial x_2}\right)+\mu\left(\frac{\partial u_1}{\partial x_2} \frac{\partial \overline{\psi}_1}{\partial x_2}+\frac{\partial u_2}{\partial x_1} \frac{\partial \overline{\psi}_2}{\partial x_1}\right) \\
	&\qquad +\lambda\left(\frac{\partial u_2}{\partial x_2} \frac{\partial \overline{\psi}_1}{\partial x_1}+\frac{\partial u_1}{\partial x_1} \frac{\partial \overline{\psi}_2}{\partial x_2}\right)+\mu\left(\frac{\partial u_2}{\partial x_1} \frac{\partial \overline{\psi}_1}{\partial x_2}+\frac{\partial u_1}{\partial x_2} \frac{\partial \overline{\psi}_2}{\partial x_1}\right).
\end{align*}
Note that $A: B=\operatorname{tr}\left(A B^{\top}\right)$ is the Frobenius inner product of 2$\times$2 matrices $A$ and $B$. Besides, from the second Korn's inequality, it can be easily verified that $\|\cdot\|_{\mathscr{H}^1_\mathrm{qp}(\Omega)}$ is equivalent to the standard $H^1$-norm; see also \cite{Xu21}. 

By the Green's formula and the first Betti's formula, we can derive the weak formulation of the original problem (\ref{DtNproblem}): Given $p^\mathrm{in}$, seek $\boldsymbol{U}=\left(p^\mathrm{sc}, \boldsymbol{u}\right) \in \mathscr{H}^1_\mathrm{qp}(\Omega)$ such that
\begin{align}\label{DtN varproblem}
	A(\boldsymbol{U}, \boldsymbol{V})=L(\boldsymbol{V}),\quad \forall~ \boldsymbol{V}=(\varphi, \boldsymbol{\psi}) \in \mathscr{H}^1_\mathrm{qp}(\Omega).
\end{align}
Here, the sesquilinear form $A: \mathscr{H}^1_\mathrm{qp}(\Omega) \times \mathscr{H}^1_\mathrm{qp}(\Omega) \rightarrow \mathbb{C}$ is  defined by 
\begin{align}\label{DtN varleft}
	A(\boldsymbol{U}, \boldsymbol{V})
	=\widetilde{A}(\boldsymbol{U}, \boldsymbol{V})
	+ \widetilde{B}_1(\boldsymbol{U}, \boldsymbol{V})
	+ \widetilde{B}_2(\boldsymbol{U}, \boldsymbol{V}),
\end{align}
where
\begin{align*}
	& \widetilde{A}(\boldsymbol{U}, \boldsymbol{V})=\int_{\Omega_+}\left(\nabla p^\mathrm{sc} \cdot \nabla \overline{\varphi}-\kappa^2 p^\mathrm{sc} \overline{\varphi}\right) \mathrm{d} \boldsymbol{x}
	+ \int_{\Omega_-}\left(\mathcal{E}_{\lambda,    \mu}(\boldsymbol{u}, \boldsymbol{\psi})- \omega^2\rho \boldsymbol{u} \cdot \overline{\boldsymbol{\psi}}\right) \mathrm{d} \boldsymbol{x},\\
	& \widetilde{B}_1(\boldsymbol{U}, \boldsymbol{V})
	= \int_{\Gamma} p^\mathrm{sc} \boldsymbol{n} \cdot \overline{\boldsymbol{\psi}} \mathrm{~d}s
	+ \int_{\Gamma} \rho_f \omega^2 \boldsymbol{u} \cdot \boldsymbol{n} \overline{\varphi} \mathrm{~d}s, \\
	& \widetilde{B}_2(\boldsymbol{U}, \boldsymbol{V})=-\int_{\Gamma_+} \mathscr{T}_+ p^\mathrm{sc} \overline{\varphi} \mathrm{~d}s
	- \int_{\Gamma_-} \mathscr{T}_- \boldsymbol{u} \cdot \overline{\boldsymbol{\psi}} \mathrm{~d}s.
\end{align*}
The term $L$ on the right-hand side of equation (\ref{DtN varproblem}) is a bounded linear functional dependent on  $\left(\partial_{\bm{n}} p^\mathrm{in}, p^\mathrm{in} \right) \in H^{-1 / 2}(\Gamma) \times H^{1 / 2}(\Gamma)$, which is defined by
\begin{align}
	L(\boldsymbol{V})=\int_{\Gamma} \partial_{\boldsymbol{n}} p^\mathrm{in} \overline{\varphi}-p^\mathrm{in} \boldsymbol{n} \cdot \overline{\boldsymbol{\psi}} \mathrm{~d}s.
\end{align}
We get from Lemma 4.1 in \cite{Hu16} that the sesquilinear form $A(\cdot, \cdot)$ satisfies  G\aa{}rding's inequality. Besides, it has also been shown in \cite{Hu16} that the variational problem (\ref{DtN varproblem}) has the unique solvability with Jones frequencies excluded. Then, it follows from the idea of  \cite{Hsiao04} that the following inf-sup condition holds
\begin{align}\label{dtn inf-sup}
	\sup _{0 \neq \boldsymbol{V} \in \mathscr{H}^1_\mathrm{qp}(\Omega)} \frac{|A(\boldsymbol{U}, \boldsymbol{V})|}{\|\boldsymbol{V}\|_{\mathscr{H}^1_\mathrm{qp}(\Omega)}} \geq \gamma_0\|\boldsymbol{U}\|_{\mathscr{H}^1_\mathrm{qp}(\Omega)}, \quad \forall~ \boldsymbol{U} \in \mathscr{H}^1_\mathrm{qp}(\Omega),
\end{align}
where $\gamma_0>0$ is a constant indepedent of $U$.

\section{PML Approximation}\label{SECTION3}
This section firstly focuses on constructing a truncated PML problem as an approximation to the original acoustic-elastic interaction in periodic structures, and then discusses the exponential convergence of the PML approximation error.
\subsection{PML setting}
We begin by introducing the absorbing PML media. Let the regions above  $\Gamma_{+}$ and below $\Gamma_{-}$ be filled with PML layers of thicknesses $\delta_1$ and $\delta_2$, respectively. The acoustic and elastic PML domains are respectively denoted by
\begin{align*}
	\Omega^{\mathrm{PML}}_+=\left\{\boldsymbol{x} \in \mathbb{R}^2: 0<x_1<\Lambda, h_1<x_2<h_1+\delta_1\right\},\\
	\Omega^{\mathrm{PML}}_-=\left\{\boldsymbol{x} \in \mathbb{R}^2: 0<x_1<\Lambda, h_2-\delta_2<x_2<h_2\right\},
\end{align*}
with the corresponding acoustic and elastic PML boundaries
\begin{align*}
	& \Gamma_+^{\mathrm{PML}}=\left\{x \in \mathbb{R}^2: 0<x_1<\Lambda, x_2=h_1+\delta_1\right\}, \\ 
	& \Gamma_-^{\mathrm{PML}}=\left\{x \in \mathbb{R}^2: 0<x_1<\Lambda, x_2=h_2-\delta_2\right\}.
\end{align*}

\begin{figure}[htbp]
	\centerline{\includegraphics[width=0.5\textwidth]{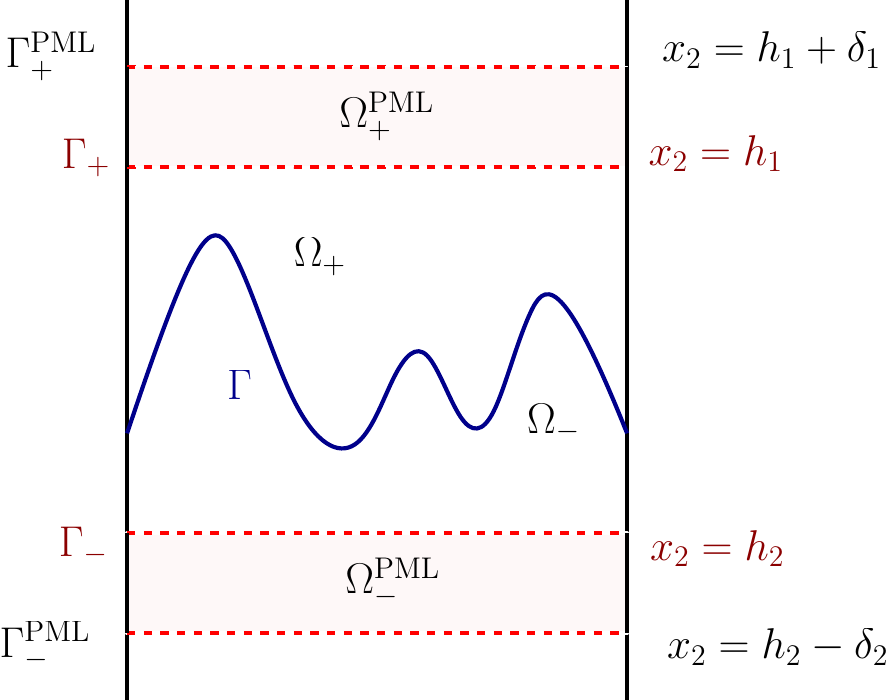}}
	\caption{Problem geometry of acoustic-elastic interaction in periodic structures with PML layers.\label{model2}}
\end{figure}

Fig. \ref{model2} shows the problem geometry of acoustic-elastic interaction in periodic structures with PML layers. Let $s(\tau)=s_1(\tau)+\mathrm{i} s_2(\tau)$ be the continuous PML function satisfying
\begin{align*}
	& s(\tau)=1 \quad \text{for} \quad h_2<\tau<h_1, \\
	& s_1(\tau) \geq 1, \quad s_2(\tau)>0 \quad \text{otherwise}.
\end{align*}
Then the PML can be introduced by the following complex coordinate stretching
\begin{align*}
	\hat{x}_2=\int_0^{x_2} s(\tau) \mathrm{~d} \tau.
\end{align*}
Obviously, we have $\nabla_{\hat{\boldsymbol{x}}}=\left[\partial_{x_1}, \partial_{\hat{x}_2}\right]^{\top}=\left[\partial_{x_1}, s^{-1} \partial_{x_2}\right]^{\top}$, where $\hat{\boldsymbol{x}}=(x_1, \hat{x}_2)^{\top}$. In addition, we can obtain
\begin{align*}
	\Delta_{\hat{\boldsymbol{x}}} p^\mathrm{sc}(\hat{\boldsymbol{x}}) + \kappa^2 p^\mathrm{sc}(\hat{\boldsymbol{x}})=0 &\quad \text { in } \Omega_\mathrm{u},\\
	\Delta^*_{\hat{\boldsymbol{x}}}\boldsymbol{u}(\hat{\boldsymbol{x}}) +  \omega^2\rho \boldsymbol{u}(\hat{\boldsymbol{x}}) = 0 &\quad \text { in } \Omega_\mathrm{d},
\end{align*}
where
\begin{align*}
	\Delta_{\hat{\boldsymbol{x}}} = \nabla_{\hat{\boldsymbol{x}}} \cdot \nabla_{\hat{\boldsymbol{x}}},\quad
	\Delta^*_{\hat{\boldsymbol{x}}} = \mu \Delta_{\hat{\boldsymbol{x}}}+(\lambda+\mu) \nabla_{\hat{\boldsymbol{x}}} \nabla_{\hat{\boldsymbol{x}}} \cdot.
\end{align*}
Denote by
\begin{align}\label{new flied}
	p^\mathrm{sc}(\hat{\boldsymbol{x}})=p^\mathrm{sc}(\hat{\boldsymbol{x}}(\boldsymbol{x})):=\hat{p}^\mathrm{sc}(\boldsymbol{x})\quad 
	\text{and}\quad 
	\boldsymbol{u}(\hat{\boldsymbol{x}})=\boldsymbol{u}(\hat{\boldsymbol{x}}(\boldsymbol{x})):=\hat{\boldsymbol{u}}(\boldsymbol{x}).
\end{align}  
It is obvious that $\hat{p}^\mathrm{sc}(\boldsymbol{x})=  p^\mathrm{sc}(\boldsymbol{x}) \text { in } \Omega_+$ and $\hat{\boldsymbol{u}}(\boldsymbol{x})=\boldsymbol{u}(\boldsymbol{x}) \text { in } \Omega_-$ since $\hat{\boldsymbol{x}}=\boldsymbol{x}$ in $x_2 \in (h_2, h_1)$.
Denote the respective acoustic and elastic PML equations by
\begin{align*} 
	\mathcal{L}_1\hat{p}^\mathrm{sc}=0 & \quad \text { in } \Omega_+^{\mathrm{PML}}, \\ \mathcal{L}_2 \hat{\boldsymbol{u}} =0 & \quad \text { in } \Omega_-^{\mathrm{PML}}.
\end{align*}
Here, the first operator $\mathcal{L}_1$ is defined by
\begin{align*}
	\mathcal{L}_1 p:=&\nabla \cdot(\mathbb{A} \nabla p) + \kappa^2 s\left(x_2\right) p,
\end{align*}
where
$$
\mathbb{A}(x)=\left[\begin{array}{cc}
	\mathbb{A}_{11} & 0 \\
	0 & \mathbb{A}_{22}
\end{array}\right]=\left[\begin{array}{cc}
	s\left(x_2\right) & 0 \\
	0 & s^{-1}(x_2)
\end{array}\right].
$$
The second operator $\mathcal{L}_2$ is defined by
\begin{align*}
	\mathcal{L}_2 \boldsymbol{u}:&=\left[\begin{array}{l}
		(2 \mu + \lambda) \frac{\partial }{\partial x_1}\left(s(x_2) \frac{\partial u_1}{\partial x_1}\right)
		+\mu \frac{\partial }{\partial x_2}\left(s^{-1}(x_2) \frac{\partial u_1}{\partial x_2}\right)
		+ (\lambda + \mu) \frac{\partial^2 u_2}{\partial x_1 \partial x_2}+\omega^2 \rho s(x_2) u_1 \\
		(2 \mu + \lambda) \frac{\partial }{\partial x_2}\left(s^{-1}(x_2) \frac{\partial u_2}{\partial x_2}\right)
		+ \mu  \frac{\partial }{\partial x_1}\left(s(x_2) \frac{\partial u_2}{\partial x_1}\right)
		+ (\lambda+\mu) \frac{\partial^2 u_1}{\partial x_1 \partial x_2} +\omega^2 \rho s(x_2) u_2
	\end{array}\right].
\end{align*}
From equations (\ref{ps rexp}), (\ref{urexp}) and (\ref{new flied}),  we can conclude that $\hat{p}^\mathrm{sc}$ and $\hat{\boldsymbol{u}}$ exhibit exponential decay properties in $\Omega^\mathrm{e}_+$ and $\Omega^\mathrm{e}_-$ as $x_2 \rightarrow \infty$ and $x_2 \rightarrow-\infty$, respectively. Thus, it is reasonable to impose the homogeneous Dirichlet boundary conditions on boundaries $\Gamma_+^{\mathrm{PML}}$ and $\Gamma_-^{\mathrm{PML}}$ for the PML problem. Let  $\mathrm{D}_+=\Omega_+ \cup \Omega_+^{\mathrm{PML}} \cup \Gamma_+$ and $\mathrm{D}_-=\Omega_- \cup \Omega_-^{\mathrm{PML}} \cup \Gamma_-$.
Then the truncated PML problem reads as: Given $p^\mathrm{in}$, seek the quasi-periodic functions $\hat{p}^\mathrm{sc}$ and $\hat{\boldsymbol{u}}$ such that 
\begin{align}\label{trun PMLproblem}
	\mathcal{L}_1\hat{p}^\mathrm{sc} =0  &\quad \text { in } \mathrm{D}_+, \notag\\
	\mathcal{L}_2 \hat{\boldsymbol{u}} =0 & \quad\text { in } \mathrm{D}_-, \notag\\ 
	\hat{p}^\mathrm{sc} = 0 &\quad \text { on } \Gamma_+^{\mathrm{PML}},\\ \hat{\boldsymbol{u}} =0 & \quad \text { on }  \Gamma_-^{\mathrm{PML}}, \notag\\ 
	\partial_{\boldsymbol{n}} (p^\mathrm{in}+\hat{p}^\mathrm{sc}) =\rho_f \omega^2 \hat{\boldsymbol{u}} \cdot \boldsymbol{n} &\quad \text { on } \Gamma,\notag\\
	-(p^\mathrm{in}+\hat{p}^\mathrm{sc}) \boldsymbol{n} = \boldsymbol{T} \hat{\boldsymbol{u}} &\quad \text { on } \Gamma. \notag
\end{align}
For convenience in defining the variational problem with respect to (\ref{trun PMLproblem}), we introduce the following Sobolev spaces
\begin{align*}
	&H_{0, \mathrm{qp}}^1\left(\mathrm{D}_+\right)=\left\{\varphi \in H_{\mathrm{qp}}^1\left(\mathrm{D}_+\right): \varphi=0\right. \left.\text { on } \Gamma_+^{\mathrm{PML}} \right\}, \\
	&H_{0, \mathrm{qp}}^1\left(\mathrm{D}_-\right)^2=\left\{\boldsymbol{\psi} \in H_{\mathrm{qp}}^1\left(\mathrm{D}_-\right)^2: \boldsymbol{\psi}=0\right.  \left.\text { on } \Gamma_-^{\mathrm{PML}} \right\}.
\end{align*}
Let $\mathrm{D} = \mathrm{D}_+ \cup \mathrm{D}_-$, and define the product Sobolev space
\begin{align*}
	\mathscr{H}_{0,\mathrm{qp}}^1(\mathrm{D}):=H_{\mathrm{0,qp}}^1\left(\mathrm{D}_+\right) \times H_{0,\mathrm{qp}}^1\left(\mathrm{D}_-\right)^2=\left\{\boldsymbol{V}=(\varphi, \boldsymbol{\psi}): \varphi \in H_{\mathrm{0,qp}}^1\left(\mathrm{D}_+\right), \boldsymbol{\psi} \in H_{\mathrm{0,qp}}^1\left(\mathrm{D}_-\right)^2\right\}.
\end{align*}
Similarly, by combining the Green's formula and the first Betti's formula, we arrive at the variational formulation of (\ref{trun PMLproblem}): Given $p^\mathrm{in}$, seek $\hat{\boldsymbol{U}}=\left(\hat{p}^\mathrm{sc}, \hat{\boldsymbol{u}}\right) \in \mathscr{H}^1_\mathrm{qp}(\mathrm{D})$ such that  
\begin{align}\label{trun PMLvarproblem}
	B_\mathrm{D}(\hat{\boldsymbol{U}}, \boldsymbol{V})=L(\boldsymbol{V}), \quad \forall~ \boldsymbol{V}=(\varphi, \boldsymbol{\psi}) \in \mathscr{H}_{0,\mathrm{qp}}^1(\mathrm{D)}.
\end{align}
For any $\mathrm{G}_+ \subseteq \mathrm{D}_+$ and $\mathrm{G}_- \subseteq  \mathrm{D}_-$, we denote $\mathrm{G} = \mathrm{G}_+ \cup \mathrm{G}_-$. Then the sesquilinear form $B_\mathrm{G}: \mathscr{H}^1_\mathrm{qp}(\mathrm{G}) \times \mathscr{H}^1_\mathrm{qp}(\mathrm{G}) \rightarrow \mathbb{C}$ is defined by
\begin{align}\label{trun PMLvarleft}
	\mathrm{B}_\mathrm{G}(\boldsymbol{U}, \boldsymbol{V})
	= \widetilde{B}_\mathrm{G}(\boldsymbol{U}, \boldsymbol{V})
	+ \widetilde{B}_1(\boldsymbol{U}, \boldsymbol{V}),\quad \forall~ \boldsymbol{U}, \boldsymbol{V} \in \mathscr{H}^1_\mathrm{qp}(\mathrm{G}),
\end{align}
where
\begin{align*}
	& \widetilde{B}_\mathrm{G}(\boldsymbol{U}, \boldsymbol{V})
	= \int_{\mathrm{G}_+}\left(\mathbb{A} \nabla {p}^\mathrm{sc} \cdot \nabla \overline{\varphi} -\kappa^2 s(x_2) p^\mathrm{sc} \overline{\varphi}\right) \mathrm{d} \boldsymbol{x}
	+ \int_{\mathrm{G}_-}\left(\mathcal{S}_{\lambda, \mu}(\boldsymbol{u}, \boldsymbol{\psi})- \omega^2\rho s(x_2) \boldsymbol{u} \cdot \overline{\boldsymbol{\psi}}\right) \mathrm{d} \boldsymbol{x},\\
	& \widetilde{B}_1(\boldsymbol{U}, \boldsymbol{V})
	= \int_{\Gamma} p^\mathrm{sc} \boldsymbol{n} \cdot \overline{\boldsymbol{\psi}} \mathrm{~d}s
	+ \int_{\Gamma} \rho_f \omega^2 \boldsymbol{u} \cdot \boldsymbol{n} \overline{\varphi} \mathrm{~d}s
\end{align*}
with
\begin{align*}
	\mathcal{S}_{\lambda, \mu}(\boldsymbol{u}, \boldsymbol{\psi})
	= & (2 \mu + \lambda)\left(s(x_2)\frac{\partial u_1}{\partial x_1} \frac{\partial \overline{\psi}_1}{\partial x_1}+s^{-1}(x_2)\frac{\partial u_2}{\partial x_2} \frac{\partial \overline{\psi}_2}{\partial x_2}\right)
	+\mu\left(s^{-1}(x_2)\frac{\partial u_1}{\partial x_2} \frac{\partial \overline{\psi}_1}{\partial x_2}+s(x_2)\frac{\partial u_2}{\partial x_1} \frac{\partial \overline{\psi}_2}{\partial x_1}\right) \\
	& +\lambda\left(\frac{\partial u_2}{\partial x_2} \frac{\partial \overline{\psi}_1}{\partial x_1}+\frac{\partial u_1}{\partial x_1} \frac{\partial \overline{\psi}_2}{\partial x_2}\right)+\mu\left(\frac{\partial u_2}{\partial x_1} \frac{\partial \overline{\psi}_1}{\partial x_2}+\frac{\partial u_1}{\partial x_2} \frac{\partial \overline{\psi}_2}{\partial x_1}\right).
\end{align*}

\subsection{TBCs of PML model}
In this subsection, we aim to reformulate the problem (\ref{trun PMLvarproblem}) defined in $\mathrm{D}$ into an equivalent weak formulation defined in $\Omega$ by introducing TBCs for the truncated PML model. Let
\begin{align}\label{3.2.1}
	\eta_1=\int_{h_1}^{h_1+\delta_1} s(\tau) \mathrm{~d} \tau,\quad
	\eta_2=\int_{h_2-\delta_2}^{h_2} s(\tau) \mathrm{~d} \tau.
\end{align}
In the complex coordinate,  we have 
\begin{align}
	\Delta_{\hat{\boldsymbol{x}}} \hat{p}^\mathrm{sc} + \kappa^2 \hat{p}^\mathrm{sc}=0 &\quad \text { in }  \Omega^{\mathrm{PML}}_+, \label{p pml}\\
	\Delta^*_{\hat{\boldsymbol{x}}}\hat{\boldsymbol{u}} +  \omega^2\rho\hat{\boldsymbol{u}} = 0 &\quad \text { in }  \Omega^{\mathrm{PML}}_-.\label{u pml}
\end{align}
From \cite{Chen03}, the TBC  on $\Gamma_+$ of the PML problem is defined by
\begin{align}\label{p PML-TBC}
	\mathscr{T}_+^{\mathrm{PML}} \hat{p}^\mathrm{sc} :=\sum_{n \in Z} \mathrm{i} \beta_n \operatorname{coth}\left(-\mathrm{i} \beta_n \eta_1\right) \hat{p}^\mathrm{sc}_{n}(h_1) e^{\mathrm{i}\alpha_n x_1},
\end{align}
where $\operatorname{coth}(y)=\left(\mathrm{e}^y+\mathrm{e}^{-y}\right) /\left(\mathrm{e}^y-\mathrm{e}^{-y}\right)$. 
Next we turn to the derivation of the TBC defined on $\Gamma_-$. As is known, the solution of (\ref{u pml}) admits the Helmholtz decomposition
\begin{align}\label{PML decomp}
	\hat{\boldsymbol{u}}=\nabla_{\hat{\boldsymbol{x}}} \hat{\phi}_1+\textbf{curl} _{\hat{\boldsymbol{x}}} \hat{\phi}_2,
\end{align}
where 
$\textbf{curl}_{\hat{\boldsymbol{x}}}=\left[\partial_{\hat{x}_2},-\partial_{x_1}\right]^{\top}$. Since the solution $\hat{\boldsymbol{u}}$ is quasi-periodic, the function $\hat{\phi}_j(\boldsymbol{x})=\phi_j(\hat{\boldsymbol{x}}),~ j = 1,2$ satisfies the Helmholtz equation
\begin{align}\label{4.1}
	\Delta_{\hat{\boldsymbol{x}}} \hat{\phi}_j+\kappa_j^2 \hat{\phi}_j=0
\end{align}
with the following Fourier series expansion 
\begin{align}\label{4.2}
	\hat{\phi}_j(x_1, x_2)=\sum_{n \in \mathbb{Z}} \hat{\phi}_n^{(j)}(x_2) \mathrm{e}^{\mathrm{i} \alpha_n x_1}.
\end{align}
Combining (\ref{4.1}) and (\ref{4.2}) gives 
\begin{align*}
	s^{-1}(x_2) \frac{\mathrm{~d}}{\mathrm{~d} x_2}\left(s^{-1}(x_2) \frac{\mathrm{~d}}{\mathrm{~d} x_2} \hat{\phi}_n^{(j)}(x_2)\right)+\left(\beta^{(j)}_{n}\right)^2 \hat{\phi}_n^{(j)}(x_2)=0,
\end{align*}
whose general solution is given by 
\begin{align}\label{phisolution}
	\hat{\phi}_n^{(j)}(x_2)=M_n^{(j)} \mathrm{e}^{\mathrm{i} \beta_n^{(j)} \int_{x_2}^{h_2} s(\tau) \mathrm{d} \tau}+N_n^{(j)} \mathrm{e}^{-\mathrm{i} \beta_n^{(j)} \int_{x_2}^{h_2} s(\tau) \mathrm{d} \tau} .
\end{align}
By using equations (\ref{PML decomp}), (\ref{4.2}), (\ref{phisolution}) and the homogeneous Dirichlet boundary condition
\begin{align*}
	\hat{\boldsymbol{u}}(x_1, h_2-\delta_2)=0 \quad \text { on } \Gamma_-^{\mathrm{PML}},
\end{align*}
we can see that the coefficients  $M_n^{(j)}$  and  $N_n^{(j)} $  satisfy the linear equations 
\begin{align*}
	&\left[\begin{array}{cccc}
		\alpha_n & \alpha_n & -\beta_n^{(2)} & \beta_n^{(2)} \\
		-\beta_n^{(1)} & \beta_n^{(1)} & -\alpha_n & -\alpha_n \\
		\alpha_n \mathrm{e}^{\mathrm{i} \beta_n^{(1)} \eta_2} & \alpha_n \mathrm{e}^{-\mathrm{i} \beta_n^{(1)} \eta_2} & -\beta_n^{(2)} \mathrm{e}^{\mathrm{i} \beta_n^{(2)} \eta_2} & \beta_n^{(2)} \mathrm{e}^{-\mathrm{i} \beta_n^{(2)} \eta_2}  \\
		-\beta_n^{(1)} \mathrm{e}^{\mathrm{i} \beta_n^{(1)} \eta_2} & \beta_n^{(1)} \mathrm{e}^{-\mathrm{i} \beta_n^{(1)} \eta_2} & -\alpha_n \mathrm{e}^{\mathrm{i} \beta_n^{(2)} \eta_2} & -\alpha_n \mathrm{e}^{-\mathrm{i} \beta_n^{(2)} \eta_2}
	\end{array}\right]\left[\begin{array}{c}
		M_n^{(1)} \\
		N_n^{(1)} \\
		M_n^{(2)} \\
		N_n^{(2)}
	\end{array}\right]=\left[\begin{array}{c}
		-\mathrm{i} \hat{u}_n^{(1)}(h_2) \\
		-\mathrm{i} \hat{u}_n^{(2)}(h_2) \\
		0 \\
		0
	\end{array}\right].
\end{align*}
By solving the above linear equations, one has
\begin{align*}
	\begin{aligned}
		M_n^{(1)} &= \frac{\mathrm{i}}{ \mathcal{X}_{n} \hat{\mathcal{X}}_{n}}\left\{-\frac{\mathcal{X}_{n}}{2}\left(\varsigma_n^{(1)}+2\right)\left(\alpha_n \hat{u}_n^{(1)}(h_2)-\beta_n^{(2)} \hat{u}_n^{(2)}(h_2)\right)\right. \\
		& ~~ + \left.\left(\varsigma_n^{(1)}+2 \xi_n^{(1)}\right)\left(\xi_n^{(2)}-\vartheta_{n}+1\right)\left(\alpha_n \beta_n^{(1)}\beta_n^{(2)} \hat{u}_n^{(1)}(h_2)-\alpha_n^2\beta_n^{(2)} \hat{u}_n^{(2)}(h_2)\right)\right\}, \\
		N_n^{(1)} &= \frac{\mathrm{i}}{ \mathcal{X}_{n} \hat{\mathcal{X}}_{n}}\left\{\frac{\mathcal{X}_{n}\varsigma_n^{(1)} }{2}\left(\alpha_n \hat{u}_n^{(1)}(h_2)+\beta_n^{(2)} \hat{u}_n^{(2)}(h_2)\right)\right. \\
		&~~ +\left.\left(\varsigma_n^{(1)} \xi_n^{(2)}+2\xi_n^{(1)} \xi_n^{(2)}+2\xi_n^{(1)}\right)\left(\alpha_n \beta_n^{(1)} \beta_n^{(2)} \hat{u}_n^{(1)}(h_2)+\alpha_n^2 \beta_n^{(2)} \hat{u}_n^{(2)}(h_2)\right)\right\}, \\
		M_n^{(2)} &= \frac{\mathrm{i}}{ \mathcal{X}_{n} \hat{\mathcal{X}}_{n}}\left\{\frac{\mathcal{X}_{n}}{2}\left[\varsigma_n^{(1)} \vartheta_{n}-2\left(\varsigma_n^{(1)}+1\right)\left(\xi_n^{(2)}+1\right)\right]\left(-\beta_n^{(1)} \hat{u}_n^{(1)}(h_2)-\alpha_n \hat{u}_n^{(2)}(h_2)\right)\right. \\
		&~~+\left.\varsigma_n^{(1)}\left(\xi_n^{(2)}-\vartheta_{n}+1\right)\left(-\left(\beta_n^{(1)}\right)^2 \beta_n^{(2)} \hat{u}_n^{(1)}(h_2)-\alpha_n^3 \hat{u}_n^{(2)}(h_2)\right)\right\}, \\
		N_n^{(2)} &= \frac{\mathrm{i}}{ \mathcal{X}_{n} \hat{\mathcal{X}}_{n}}\left\{\frac{\mathcal{X}_{n}}{2}\left[2 \xi_n^{(2)}\left(\varsigma_n^{(1)}+1\right)-\varsigma_n^{(1)} \vartheta_{n}\right]\left(-\beta_n^{(1)} \hat{u}_n^{(1)}(h_2)+\alpha_n \hat{u}_n^{(2)}(h_2)\right)\right. \\
		&~~ -\left. \xi_n^{(2)}\left(\varsigma_n^{(1)}+2\right)\left(-\left(\beta_n^{(1)}\right)^2 \beta_n^{(2)} \hat{u}_n^{(1)}(h_2)+\alpha_n^3 \hat{u}_n^{(2)}(h_2)\right)\right\},
	\end{aligned}
\end{align*}
where
\begin{align*}
	\varsigma_n^{(j)}  &=\operatorname{coth}\left(-\mathrm{i} \beta_n^{(j)} \eta_2\right)-1, \\ 
	\xi_n^{(j)} & =\left(\mathrm{e}^{\mathrm{i} \beta_n^{(2)} \eta_2}-\mathrm{e}^{\mathrm{i} \beta_n^{(1)} \eta_2}\right) /\left(\mathrm{e}^{-\mathrm{i} \beta_n^{(j)} \eta_2}-\mathrm{e}^{\mathrm{i} \beta_n^{(j)} \eta_2}\right), \\ 
	\vartheta_{n} &=\left(\mathrm{e}^{-\mathrm{i} \beta_n^{(1)}\eta_2} -\mathrm{e}^{\mathrm{i} \beta_n^{(1)} \eta_2}\right) /\left(\mathrm{e}^{-\mathrm{i} \beta_n^{(2)} \eta_2}-\mathrm{e}^{\mathrm{i} \beta_n^{(2)} \eta_2}\right),\\
	\hat{\mathcal{X}}_{n}&=\mathcal{X}_{n}+4\alpha_n^2 \beta_n^{(1)} \beta_n^{(2)}\left(\xi_n^{(2)}-\xi_n^{(1)}-\xi_n^{(1)} \xi_n^{(2)}\right)  / \mathcal{X}_{n}.
\end{align*}
It follows from (\ref{PML decomp}) that
\begin{align}\label{4.5}
	\hat{\boldsymbol{u}}(x_1, x_2)&=\mathrm{i} \sum_{n \in \mathbb{Z}}\left[\begin{array}{c}
		\alpha_n \\
		-\beta_n^{(1)}
	\end{array}\right] M_n^{(1)} \mathrm{e}^{\mathrm{i}\left(\alpha_n x_1+\beta_n^{(1)} \int_{x_2}^{h_2} s(\tau) \mathrm{d} \tau\right)}+\left[\begin{array}{c}
		\alpha_n \\
		\beta_n^{(1)}
	\end{array}\right] N_n^{(1)} \mathrm{e}^{\mathrm{i}\left(\alpha_n x_1-\beta_n^{(1)} \int_{x_2}^{h_2} s(\tau) \mathrm{d} \tau\right)} \notag  \\  
	&~~ -{\left[\begin{array}{c}
			\beta_n^{(2)} \\
			\alpha_n
		\end{array}\right] M_n^{(2)} \mathrm{e}^{\mathrm{i}\left(\alpha_n x_1+\beta_n^{(2)} \int_{x_2}^{h_2} s(\tau) \mathrm{d} \tau\right)}+\left[\begin{array}{c}
			\beta_n^{(2)} \\
			-\alpha_n
		\end{array}\right] N_n^{(2)} \mathrm{e}^{\mathrm{i}\left(\alpha_n x_1-\beta_n^{(2)} \int_{x_2}^{h_2} s(\tau) \mathrm{d} \tau\right).}}
\end{align}
Using (\ref{uboun oper}) and (\ref{4.5}) yields the TBC for the transmitted field $\hat{\boldsymbol{u}}$ as
\begin{align}\label{u PML-TBC}
	\mathscr{B} \hat{\boldsymbol{u}}=\mathscr{T}_-^{\mathrm{PML}} \hat{\boldsymbol{u}}:=\sum_{n \in \mathbb{Z}} \hat{W}^{(n)} \hat{\boldsymbol{u}}_{n}(h_2) \mathrm{e}^{\mathrm{i} \alpha_n x_1},
\end{align}
where $\hat{W}^{(n)}=\left[\begin{array}{ll}
	\hat{w}_{11}^{(n)} & \hat{w}_{12}^{(n)} \\
	\hat{w}_{21}^{(n)} & \hat{w}_{22}^{(n)}
\end{array}\right]$ is a $2 \times 2$ matrix with entries defined by 
\begin{align*}
	& \hat{w}_{11}^{(n) }= \frac{\mathrm{i}}{ \mathcal{X}_{n} \hat{\mathcal{X}}_{n}}\left\{
	\omega^2 \rho \beta_n^{(1)}\mathcal{X}_{n}+ \omega^2\rho \beta_n^{(1)}\left[\varsigma_n^{(1)} \alpha_n^2+\left(\varsigma_n^{(1)} \vartheta_{n}+2 \xi_n^{(2)}\right) \beta_n^{(1)} \beta_n^{(2)}\right]\right\}, \\
	& \hat{w}_{12}^{(n)} = \frac{\mathrm{i}\alpha_n}{ \mathcal{X}_{n} \hat{\mathcal{X}}_{n}}\left\{
	- 2\mu \mathcal{X}_{n} \hat{\mathcal{X}}_{n} 
	+ \omega^2\rho \mathcal{X}_{n} 
	+ \omega^2\rho \beta_n^{(1)}\beta_n^{(2)}   \left[\varsigma_n^{(1)}\left(2 \xi_n^{(2)}
	-\vartheta_{n}+1\right)+2 \xi_n^{(2)}\right] \right\}, \\
	& \hat{w}_{21}^{(n)} = \frac{\mathrm{i}\alpha_n}{ \mathcal{X}_{n} \hat{\mathcal{X}}_{n}} \left\{
	2\mu\mathcal{X}_{n} \hat{\mathcal{X}}_{n} 
	- \omega^2 \rho \mathcal{X}_{n}
	+ \omega^2\rho  \beta_n^{(1)} \beta_n^{(2)} 
	\left[\varsigma_n^{(1)}\left(2 \xi_n^{(2)}-\vartheta_{n}+1\right)+ 4\xi_n^{(1)}\left(  \xi_n^{(2)} + 1\right) -2\xi_n^{(2)} \right] \right\}, \\
	& \hat{w}_{22}^{(n)}=\frac{\mathrm{i}}{ \mathcal{X}_{n} \hat{\mathcal{X}}_{n}} \left\{
	\omega^2 \rho \beta_n^{(2)}\mathcal{X}_{n}
	+ \omega^2 \rho \beta_n^{(2)}
	\left[\left(\varsigma_n^{(1)} \vartheta_{n}+2 \xi_n^{(2)}\right) \alpha_n^2 + \varsigma_n^{(1)} \beta_n^{(1)} \beta_n^{(2)}\right] \right\}.
\end{align*}
Based on (\ref{p PML-TBC}) and (\ref{u PML-TBC}), we can reduce the PML model (\ref{trun PMLproblem}) to an equivalent boundary value problem: Given $p^\mathrm{in}$, seek the quasi-periodic functions $p^{\mathrm{sc},\mathrm{PML}}$ and $\boldsymbol{u}^{\mathrm{PML}}$ such that
\begin{align}\label{PMLproblem}
	\Delta p^{\mathrm{sc},\mathrm{PML}}+\kappa^2 p^{\mathrm{sc},\mathrm{PML}}=0 & \quad\text { in } \Omega_+, \notag\\
	\Delta^* \boldsymbol{u}^\mathrm{PML} + \omega^2\rho \boldsymbol{u}^\mathrm{PML}=0 & \quad\text { in } \Omega_-, \notag \\
	\partial_{\boldsymbol{n}} (p^\mathrm{in}+p^{\mathrm{sc},\mathrm{PML}}) =\rho_f \omega^2 \boldsymbol{u}^\mathrm{PML} \cdot \boldsymbol{n} & \quad\text { on } \Gamma,  \\
	-(p^\mathrm{in}+p^{\mathrm{sc},\mathrm{PML}}) \boldsymbol{n}=\boldsymbol{T} \boldsymbol{u}^\mathrm{PML} & \quad\text { on } \Gamma, \notag \\
	\partial_{x_2} p^{\mathrm{sc},\mathrm{PML}}=\mathscr{T}_+^\mathrm{PML} p^{\mathrm{sc},\mathrm{PML}} & \quad\text { on } \Gamma_+, \notag\\
	\mathscr{B} \boldsymbol{u}^\mathrm{PML}=\mathscr{T}_-^\mathrm{PML} \boldsymbol{u}^\mathrm{PML} &\quad \text { on } \Gamma_-.\notag
\end{align}
Besides, the variational formulation of the problem (\ref{PMLproblem}) reads as: Given $p^\mathrm{in}$, seek ${\boldsymbol{U}}^{\mathrm{PML}}=(p^{\mathrm{sc},\mathrm{PML}}, \boldsymbol{u}^{\mathrm{PML}}) \in \mathscr{H}^1_\mathrm{qp}(\Omega)$ such that
\begin{align}\label{PML varproblem}
	A^{\mathrm{PML}}(\boldsymbol{U}^{\mathrm{PML}}, \boldsymbol{V})
	=L(\boldsymbol{V}), \quad \forall ~ \boldsymbol{V}=(\varphi, \boldsymbol{\psi}) \in \mathscr{H}^1_\mathrm{qp}(\Omega),
\end{align}
where
\begin{align}\label{PML varleft}
	A^{\mathrm{PML}}(\boldsymbol{U}^{\mathrm{PML}}, \boldsymbol{V})
	=\widetilde{A}^{\mathrm{PML}}(\boldsymbol{U}^{\mathrm{PML}}, \boldsymbol{V})
	+\widetilde{B}_1^{\mathrm{PML}}(\boldsymbol{U}^{\mathrm{PML}}, \boldsymbol{V})
	+\widetilde{B}_2^{\mathrm{PML}}(\boldsymbol{U}^{\mathrm{PML}}, \boldsymbol{V})
\end{align}
with
\begin{align*}
	&\widetilde{A}^{\mathrm{PML}}(\boldsymbol{U}^{\mathrm{PML}}, \boldsymbol{V})
	= \int_{\Omega_+}\left(\nabla p^{\mathrm{sc},\mathrm{PML}} \cdot \nabla \overline{\varphi}-\kappa^2 p^{\mathrm{sc},\mathrm{PML}} \overline{\varphi}\right) \mathrm{d} \boldsymbol{x}
	+ \int_{\Omega_-}\left(\mathcal{E}_{\lambda, \mu}(\boldsymbol{u}^\mathrm{PML}, \boldsymbol{v})- \omega^2\rho \boldsymbol{u}^\mathrm{PML} \cdot \overline{\boldsymbol{\psi}}\right) \mathrm{d} \boldsymbol{x},\\
	&\widetilde{B}_1^{\mathrm{PML}}(\boldsymbol{U}^{\mathrm{PML}}, \boldsymbol{V})
	= \int_{\Gamma} \rho_f \omega^2 \boldsymbol{u}^\mathrm{PML} \cdot \boldsymbol{n} \overline{\varphi} \mathrm{~d}s + \int_{\Gamma} p^{\mathrm{sc},\mathrm{PML}} \boldsymbol{n} \cdot \overline{\boldsymbol{\psi}} \mathrm{~d}s, \\
	&\widetilde{B}_2^{\mathrm{PML}}(\boldsymbol{U}^{\mathrm{PML}}, \boldsymbol{V})=
	- \int_{\Gamma_+} \mathscr{T}_+^\mathrm{PML} p^{\mathrm{sc},\mathrm{PML}} \overline{\varphi} \mathrm{~d}s - \int_{\Gamma_-} \mathscr{T}_-^\mathrm{PML} \boldsymbol{u}^\mathrm{PML} \cdot \overline{\boldsymbol{\psi}} \mathrm{~d}s.
\end{align*}

The relationship between the variational formulations (\ref{trun PMLvarproblem}) and (\ref{PML varproblem}) can be established via the following lemma. The proof is direct based on our constructions of TBCs for the PML model. For brevity, we omit the detailed steps here.
\begin{lemma}
	Any solution $\hat{\boldsymbol{U}}$ of (\ref{trun PMLvarproblem}) restricted to the domain $\Omega$ is a solution of (\ref{PML varproblem}). Conversely, any solution $\boldsymbol{U}^\mathrm{PML}$ of (\ref{PML varproblem}) can be uniquely extended to the region $\mathrm{D}$ to be a solution $\hat{\boldsymbol{U}}$ of (\ref{trun PMLvarproblem}).
\end{lemma}

\subsection{Error analysis of PML approximation}
In this subsection, we mainly focus on estimating the error between $\boldsymbol{U}^{\mathrm{PML}}$ and $\boldsymbol{U}$. To begin with, we introduce some notations for facilitating the subsequent theoretical analysis.

For the upper region, let $\Theta_n=\left|\kappa^2-\alpha_n^2\right|^{1 / 2}$ and $I=\left\{n:\left|\alpha_n\right|<\kappa\right\}$. Then we obtain $\beta_n=\Theta_n $ for $n \in I$, and $\beta_n=\mathbf{i} \Theta_n$ for $n \notin I$. Similarly, for the  lower region, define $\Theta_n^{(j)}=\left|\kappa_j^2-\alpha_n^2\right|^{1 / 2}$ and $I_j=\left\{n:\left|\alpha_n\right|<\kappa_j\right\}$, $j=1,2$. Clearly, we have $\beta_n^{(j)}=\Theta_n^{(j)}$ for $n \in I_j$, and $\beta_n^{(j)}=\mathbf{i} \Theta_n^{(j)}$ for $n \notin I_j$. Denote by
\begin{align*}
	&\Theta^{i}=\min \left\{\Theta_{n}: n \in I\right\}, \quad  \Theta^{e}=\min \left\{\Theta_{n}: n \notin I\right\},\\
	&\Theta_j^{i}=\min \left\{\Theta_n^{(j)}: n \in I_j\right\}, \quad \Theta_j^{e}=\min \left\{\Theta_n^{(j)}: n \notin I_j\right\}.
\end{align*}

The following Lemma is the result about the trace theorem.
\begin{lemma}[\text{See} \cite{Chen03, Jiang17-3}]\label{trace}
	For any $\varphi \in H_{\mathrm{qp}}^1\left(\Omega_+\right)$, $\boldsymbol{\psi} \in H_{\mathrm{qp}}^1\left(\Omega_-\right)^2$, one has
	\begin{align*}
		&\|\varphi\|_{L^2(\Gamma_+)} \leq\|\varphi\|_{H^{1 / 2}(\Gamma_+)} \leq \gamma_1\|\varphi\|_{H^1(\Omega_+)},\\
		&\|\boldsymbol{\psi}\|_{L^2(\Gamma_-)^2} \leq\|\boldsymbol{\psi}\|_{H^{1/2}(\Gamma_-)^2} \leq \gamma_1\|\boldsymbol{\psi}\|_{H^1(\Omega_-)^2},
	\end{align*}
	where $\gamma_1=\left(1+(h_1-h_2)^{-1}\right)^{1 / 2}$.
\end{lemma}
The Lemma \ref{PML-DtNerror} provides the error estimate between $\mathscr{T}_m^{\mathrm{PML}}$ and $\mathscr{T}_m$ $(m=+,-)$.
\begin{lemma}\label{PML-DtNerror}
	For any $(p,\boldsymbol{u}) \in \mathscr{H}^1_\mathrm{qp}\left(\Omega\right)$, $(\varphi, \boldsymbol{\psi}) \in \mathscr{H}^1_\mathrm{qp}\left(\Omega\right)$, we obtain
	\begin{align}
		&\left|\left\langle\left(\mathscr{T}_+^{\mathrm{PML}}-\mathscr{T}_+\right) p, \varphi\right\rangle_{\Gamma_+}\right| \leq F_1\|p\|_{L^2(\Gamma_+)}\|\varphi\|_{L^2(\Gamma_+)},\label{DtNerror1}\\
		&\left|\left\langle\left(\mathscr{T}_-^{\mathrm{PML}}-\mathscr{T}_-\right) \boldsymbol{u}, \boldsymbol{\psi}\right\rangle_{\Gamma_-}\right| \leq 
		F_2\|\boldsymbol{u}\|_{L^2(\Gamma_-)^2}\|\boldsymbol{\psi}\|_{L^2(\Gamma_-)^2},\label{DtNerror2}
	\end{align}
	where 
	\begin{align*}
		F_1=\max \left\{\frac{2 \Theta^{i}}{e^{2 \Im{\eta_1} \Theta^{i}}-1}, \frac{2 \Theta^{e}}{e^{2 \Re{\eta_1} \Theta^{e}}-1}\right\}
	\end{align*}
	and
	\begin{align*}
		F_2=\frac{34 \omega^2\rho}{\kappa_1^4} \times \max _{j=1,2}\left\{\frac{\Theta_j^{i}}{\mathrm{e}^{\frac{1}{2} \Theta_j^{i} \Im{\eta_2}}-1}, \frac{\Theta_j^{e}}{\mathrm{e}^{\frac{1}{2} \Theta_j^{e} \Re{\eta_2}}-1}\right\} \times \max \left\{6\kappa_2, \kappa_2^2+4, 8\kappa_2^4, \frac{8\kappa_2^3}{\kappa_1^2},\frac{12\left(\kappa_2^2+16\right)^2}{\kappa_1^2}\right\}.
	\end{align*}
\end{lemma}
\begin{proof}
	The derivation of (\ref{DtNerror1}) is provided in \cite[Lemma 2.2]{Chen03}, and inequality (\ref{DtNerror2}) follows by an argument analogous to \cite[Lemma 3.2]{Jiang17-3}. Details are omitted here for brief.
\end{proof}
In subsequent theoretical analysis, we always take the medium function $s(x_2)$ as 
\begin{align*}
	\begin{array}{ll}
		s(x_2) = 1+\sigma_1 \left(\frac{x_2-h_1}{\delta_1}\right)^t, & \quad\text { if } x_2 \geq h_1,\quad t \geq 1, \\
		s(x_2) = 1+\sigma_2 \left(\frac{h_2-x_2}{\delta_2}\right)^t, & \quad\text { if } x_2 \leq h_2,\quad t \geq 1.
	\end{array}
\end{align*}
Based on the definition of $\eta_j$ ($j=1,2$) given by (\ref{3.2.1}), a straightforward calculation shows
$$
\Re \eta_j=\left(1+\frac{\Re \sigma_j}{t+1}\right) \delta_j, \quad \Im \eta_j=\left(\frac{\Im \sigma_j}{t+1}\right) \delta_j.
$$
which further implies that both $\Re{\eta_j}$ and $\Im{\eta_j}$ are determined by the PML parameters $\sigma_j$ and $\delta_j$. As is shown in Lemma \ref{PML-DtNerror}, the errors associated with the boundary operators decay exponentially with respect to the PML parameters.

We are now ready to estimate the error between the solution of the PML problem and that of the original problem.
\begin{theorem}\label{theorem1}
	Let $\boldsymbol{U}$ and $\boldsymbol{U}^{\mathrm{PML}}$ be respective solutions to (\ref{DtN varproblem}) and (\ref{PML varproblem}). If $\widetilde{F}\gamma_1^2<\gamma_0$, then the truncated PML  problem (\ref{PML varproblem}) has a unique solution $\boldsymbol{U}^{\mathrm{PML}}$, which satisfies the following error estimate:
	\begin{align}\label{4.7}
		\|\boldsymbol{U}-\boldsymbol{U}^{\mathrm{PML}}\|_{\mathscr{H}^1_\mathrm{qp}(\Omega)}&:=\sup _{0 \neq \boldsymbol{V} \in \mathscr{H}^1_\mathrm{qp}(\Omega)} \frac{\left|A\left(\boldsymbol{U}-\boldsymbol{U}^{\mathrm{PML}}, \boldsymbol{V}\right)\right|}{\|\boldsymbol{V}\|_{\mathscr{H}^1_\mathrm{qp}(\Omega)}} \notag\\ &\leq 
		F_1\gamma_1\left\|p^{\mathrm{sc},\mathrm{PML}}\right\|_{L^2(\Gamma_+)}+F_2\gamma_1\left\|\boldsymbol{u}^{\mathrm{PML}}\right\|_{L^2(\Gamma_-)^2},
	\end{align}
	where $\widetilde{F}=\max\left\{F_1, F_2\right\}$, and the constants $\gamma_0, \gamma_1$ are respectively presented in  (\ref{dtn inf-sup}) and Lemma \ref{trace}.
\end{theorem}
\begin{proof}
	In order to demonstrate that the variational problem (\ref{PML varproblem}) has a unique weak solution, we only need to establish the coercivity of the sesquilinear form $A^{\mathrm{PML}}$ in (\ref{PML varleft}). Based on the definition of $A(\cdot,\cdot)$ and $A^{\mathrm{PML}}(\cdot,\cdot)$, Lemmas \ref{trace} and \ref{PML-DtNerror}, we have
	\begin{align*}
		|A^{\mathrm{PML}}(\boldsymbol{U}, \boldsymbol{V})| &\geq|A(\boldsymbol{U}, \boldsymbol{V})| - \left|\left\langle\left(\mathscr{T}_+^{\mathrm{PML}}-\mathscr{T}_+\right) p^\mathrm{sc}, \varphi\right\rangle_{\Gamma_-}\right| - 
		\left|\left\langle\left(\mathscr{T}_-^{\mathrm{PML}}-\mathscr{T}_-\right) \boldsymbol{u}, \boldsymbol{\psi}\right\rangle_{\Gamma_-}\right| \\
		&\geq|A(\boldsymbol{U}, \boldsymbol{V})| - F_1\|p^\mathrm{sc}\|_{L^2(\Gamma_+)}\|\varphi\|_{L^2(\Gamma_+)}- F_2\|\boldsymbol{u}\|_{L^2(\Gamma_-)^2}\|\boldsymbol{\psi}\|_{L^2(\Gamma_-)^2}\\
		&\geq|A(\boldsymbol{U}, \boldsymbol{V})| - F_1\gamma_1^2\|p^\mathrm{sc}\|_{H^1(\Omega_+)}\|\varphi\|_{H^1(\Omega_+)}- F_2\gamma_1^2\|\boldsymbol{u}\|_{H^1(\Omega_-)^2}\|\boldsymbol{\psi}\|_{H^1(\Omega_-)^2}\\
		&\geq|A(\boldsymbol{U}, \boldsymbol{V})| - \widetilde{F}\gamma_1^2\|\boldsymbol{U}\|_{\mathscr{H}^1_\mathrm{qp}(\Omega)}\|\boldsymbol{V}\|_{\mathscr{H}^1_\mathrm{qp}(\Omega)}.
	\end{align*}
	Combining inequality (\ref{dtn inf-sup}) and the assumption $\widetilde{F}\gamma_1^2<\gamma_0$ yields
	\begin{align*}
		|A^{\mathrm{PML}}(\boldsymbol{U}, \boldsymbol{V})|
		&\geq \gamma_0 \|\boldsymbol{U}\|_{\mathscr{H}^1_\mathrm{qp}(\Omega)}\|\boldsymbol{V}\|_{\mathscr{H}^1_\mathrm{qp}(\Omega)}- \widetilde{F}\gamma_1^2\|\boldsymbol{U}\|_{\mathscr{H}^1_\mathrm{qp}(\Omega)}\|\boldsymbol{V}\|_{\mathscr{H}^1_\mathrm{qp}(\Omega)}\\
		&=(\gamma_0-\widetilde{F}\gamma_1^2)\|\boldsymbol{U}\|_{\mathscr{H}^1_\mathrm{qp}(\Omega)}\|\boldsymbol{V}\|_{\mathscr{H}^1_\mathrm{qp}(\Omega)}.
	\end{align*}
	Next, we turn to estimate (\ref{4.7}). By utilizing (\ref{DtN varproblem}), (\ref{PML varproblem}), Lemma \ref{trace} and Lemma \ref{PML-DtNerror}, one can get
	\begin{align*}
		A\left(\boldsymbol{U}-\boldsymbol{U}^{\mathrm{PML}}, \boldsymbol{V}\right)
		&=A\left(\boldsymbol{U},\boldsymbol{V}\right)-A\left(\boldsymbol{U}^{\mathrm{PML}}, \boldsymbol{V}\right)\\
		&=A\left(\boldsymbol{U},\boldsymbol{V}\right) 
		- A^{\mathrm{PML}}\left(\boldsymbol{U}^{\mathrm{PML}}, \boldsymbol{V}\right) 
		+ A^{\mathrm{PML}}\left(\boldsymbol{U}^{\mathrm{PML}}, \boldsymbol{V}\right) -A\left(\boldsymbol{U}^{\mathrm{PML}}, \boldsymbol{V}\right)\\
		&=\left\langle\left(\mathscr{T}_+-\mathscr{T}_+^{\mathrm{PML}}\right) p^{\mathrm{sc},\mathrm{PML}}, \varphi\right\rangle_{\Gamma_+}+
		\left\langle\left(\mathscr{T}_--\mathscr{T}_-^{\mathrm{PML}}\right) \boldsymbol{u}^{\mathrm{PML}}, \boldsymbol{\psi}\right\rangle_{\Gamma_-}\\
		&\leq F_1\|p^{\mathrm{sc},\mathrm{PML}}\|_{L^2(\Gamma_+)}\|\varphi\|_{L^2(\Gamma_+)}+ F_2\|\boldsymbol{u}^{\mathrm{PML}}\|_{L^2(\Gamma_-)^2}\|\boldsymbol{\psi}\|_{L^2(\Gamma_-)^2}\\
		&\leq F_1 \gamma_1 \|p^{\mathrm{sc},\mathrm{PML}}\|_{L^2(\Gamma_+)}\|\varphi\|_{H^1(\Omega_+)} + F_2 \gamma_1 \|\boldsymbol{u}^{\mathrm{PML}}\|_{L^2(\Gamma_-)^2}\|\boldsymbol{\psi}\|_{H^1(\Omega_-)^2}\\
		&\leq (F_1 \gamma_1 \|p^{\mathrm{sc},\mathrm{PML}}\|_{L^2(\Gamma_+)} + F_2 \gamma_1 \|\boldsymbol{u}^{\mathrm{PML}}\|_{L^2(\Gamma_-)^2})\|\boldsymbol{V}\|_{\mathscr{H}^1_\mathrm{qp}(\Omega)},
	\end{align*}
	which completes the proof.
\end{proof}
It follows from Theorem \ref{theorem1} that the PML approximation error can be significantly reduced by increasing the PML parameters $\sigma_j$ or $\delta_j$.

\section{Discrete Problem} \label{SECTION4}
This section focuses on the FE discretization of the PML variational problem (\ref{trun PMLvarproblem}) and investigating the corresponding residual a posteriori error estimate.

\subsection{FEM approximation}
Let $\mathscr{M}_{h}=\mathscr{M}_{h, +} \cup\mathscr{M}_{h, -}$, where $\mathscr{M}_{h, +}$ and $\mathscr{M}_{h, -}$  are regular triangulations of $\mathrm{D}_+$ and $\mathrm{D}_-$, respectively.  Assume that any triangle element $T$ must be entirely contained within the domain $\overline{\Omega_m^{\mathrm{PML}}}$ ($m=+,-$) or $\overline{\Omega}$. Let $\mathcal{S}_h^+ \subset H_{\mathrm{qp}}^1\left(\mathrm{D}_+\right)$ and $\mathcal{S}_h^-\subset H_{\mathrm{qp}}^1\left(\mathrm{D}_-\right)^2$ be the conforming finite element spaces, i.e.,
$$
\begin{aligned}
	& \mathcal{S}_h^+:=\left\{p_h \in C\left(\overline{\Omega}_+\right):\left.p_h\right|_{T} \in P_k(T),~\forall~T \in \mathscr{M}_{h, +}, ~p_h\left(\Lambda, x_2\right)=p_h\left(0, x_2\right) e^{\mathrm{i} \alpha \Lambda}\right\}, \\
	&\mathcal{S}_h^-:=\left\{\boldsymbol{u}_h \in C\left(\overline{\Omega}_-\right):\left.\boldsymbol{u}_h\right|_{T} \in P_{k}(T)^2,~\forall~ T \in \mathscr{M}_{h, -},~ \boldsymbol{u}_h\left(\Lambda, x_2\right)=\boldsymbol{u}_h\left(0, x_2\right) e^{\mathrm{i} \alpha \Lambda}\right\},
\end{aligned}
$$
where $P_k(T)$ denotes the set of all polynomials of degree no more than $k$ $\left(k \in \mathbb{N}^{+}\right)$. Let $\mathscr{H}_h^1(\mathrm{D}):=\mathcal{S}_h^{+} \times \mathcal{S}_h^{-} \subset \mathscr{H}^1(\mathrm{D})$ and $\stackrel{\circ}{\mathscr{H}_h}(\mathrm{D})=\mathscr{H}_h^1(\mathrm{D}) \cap \mathscr{H}_{0,\mathrm{qp}}^1(\mathrm{D})$. The FE approximation to \eqref{trun PMLvarproblem} is formulated as: Given $p^\mathrm{in}$, seek $\hat{\boldsymbol{U}}_h \in \mathscr{H}_h^1(\mathrm{D})$ such that $\hat{p}^{\mathrm{sc}}_h = 0 \text { on } \Gamma_+^{\mathrm{PML}}$, $\hat{\boldsymbol{u}}_h =0   \text { on }  \Gamma_-^{\mathrm{PML}}$, and
\begin{align}\label{trun APPLvarproblem}
	B_\mathrm{D}(\hat{\boldsymbol{U}}_h, \boldsymbol{V}_h)=L(\boldsymbol{V}_h), \quad \forall~\boldsymbol{V}_h=(\varphi_h, \boldsymbol{\psi}_h) \in \stackrel{\circ}{\mathscr{H}_h}(\mathrm{D}).
\end{align}
where $B_\mathrm{D}(\cdot,\cdot)$ is a sesquilinear form defined by (\ref{trun PMLvarleft}).

This paper aims to develop a residual-based a posteriori error estimate as well as the adaptive FE algorithm. Therefore, we assume that the variational problem (\ref{trun APPLvarproblem}) is uniquely solvable.

\subsection{Error representation}
In this subsection, we mainly introduce an error representation formula, which is used to facilitate the a posteriori error analysis in Theorem \ref{theorem main}. For any $\boldsymbol{V}=(\varphi,\boldsymbol{\psi}) \in \mathscr{H}^1_\mathrm{qp}\left(\Omega\right)$,  let $\tilde{\boldsymbol{V}}$ be the extension of $\boldsymbol{V}$ such that $\tilde{\boldsymbol{V}}=\boldsymbol{V}$ in $\Omega$ and $\tilde{\boldsymbol{V}}$ satisfies the following boundary value problem
\begin{align}\label{extenstion}
	\Delta_{\hat{\boldsymbol{x}}}  \overline{\tilde{\varphi}} + \kappa^2 \overline{\tilde{\varphi}}=0 & \quad\text { in } \Omega_+^{\mathrm{PML}}, \notag\\
	\Delta^*_{\hat{\boldsymbol{x}}}\overline{\tilde{\boldsymbol{\psi}}} +  \omega^2\rho \overline{\tilde{\boldsymbol{\psi}}} = 0 &\quad \text { in } \Omega^{\mathrm{PML}}_-, \notag\\
	\tilde{\varphi}\left(x_1, h_1\right)=\varphi\left(x_1, h_1\right) & \quad \text { on } \Gamma_+, \\
	\tilde{\boldsymbol{\psi}}\left(x_1, h_2\right)=\boldsymbol{\psi} \left(x_1, h_2\right) & \quad\text { on } \Gamma_-, \notag\\
	\tilde{\varphi}\left(x_1, h_1+\delta_1\right)=0 &\quad \text { on } \Gamma_+^{\mathrm{PML}},\notag \\
	\tilde{\boldsymbol{\psi}}\left(x_1, h_2-\delta_2\right)=0 &\quad \text { on } \Gamma_-^{\mathrm{PML}}.\notag
\end{align}
We give the following two lemmas, which will be used in the subsequent a posteriori estimate.
\begin{lemma}\label{lemma5}
	For any $(p,\boldsymbol{u}) \in \mathscr{H}^1_\mathrm{qp}\left(\Omega\right)$, $(\varphi, \boldsymbol{\psi}) \in \mathscr{H}^1_\mathrm{qp}\left(\Omega\right)$, one can obtain
	\begin{align*}
		\int_{\Gamma_+} \mathscr{T}_+^{\mathrm{PML}} p  \overline{\varphi} \mathrm{~d} s &=\int_{\Gamma_+} p  \partial_{x_2} \overline{\tilde{\varphi}} \mathrm{~d} s,\\
		\int_{\Gamma_-} \mathscr{T}_-^{\mathrm{PML}} \boldsymbol{u} \cdot \overline{\boldsymbol{\psi}} \mathrm{~d} s &=\int_{\Gamma_-} \boldsymbol{u} \cdot \mathscr{B} \overline{\tilde{\boldsymbol{\psi}}} \mathrm{~d} s.
	\end{align*}
\end{lemma}
\begin{proof}
	To begin with, we introduce $\hat{z} \in H_{\mathrm{qp}}^1\left(\Omega_+^{\mathrm{PML}}\right)$ and $\hat{\boldsymbol{w}} \in H_{\mathrm{qp}}^1\left(\Omega_-^{\mathrm{PML}}\right)^2$ satisfying
	\begin{align*}
		\Delta_{\hat{\boldsymbol{x}}}  \hat{z} + \kappa^2 \hat{z}=0 & \quad\text { in } \Omega_+^{\mathrm{PML}}, \notag\\
		\Delta^*_{\hat{\boldsymbol{x}}}\hat{\boldsymbol{w}} +  \omega^2\rho \hat{\boldsymbol{w}} = 0 &\quad \text { in } \Omega^{\mathrm{PML}}_-, \notag\\
		\hat{z}\left(x_1, h_1\right)=p \left(x_1, h_1\right) & \quad \text { on } \Gamma_+, \\
		\hat{\boldsymbol{w}}\left(x_1, h_2\right)=\boldsymbol{u} \left(x_1, h_2\right) & \quad\text { on } \Gamma_-, \notag\\
		\hat{z}\left(x_1, h_1+\delta_1\right)=0 &\quad \text { on } \Gamma_+^{\mathrm{PML}},\notag \\
		\hat{\boldsymbol{w}}\left(x_1, h_2-\delta_2\right)=0 &\quad \text { on } \Gamma_-^{\mathrm{PML}}.\notag
	\end{align*}
	It follows from the definitions of $\mathscr{T}_+^\mathrm{PML}$ and $\mathscr{T}_-^{\mathrm{PML}}$ that
	\begin{align*}
		\mathscr{T}_+^{\mathrm{PML}} p=\partial_{x_2}{\hat{z}}  \quad \text { on } \Gamma_+,\\
		\mathscr{T}_-^{\mathrm{PML}} \boldsymbol{u}=\mathscr{B} \hat{\boldsymbol{w}} \quad \text { on } \Gamma_-.
	\end{align*}
	By combining Green's formula and the extension, we have
	\begin{align*}
		\int_{\Gamma_+} p  \partial_{x_2} \overline{\tilde{\varphi}} \mathrm{~d} s
		& = \int_{\Gamma_+} \hat{z}  \partial_{x_2} \overline{\tilde{\varphi}} \mathrm{~d} s  = -\int_{\Omega_+^{\mathrm{PML}}}\left[\nabla_{\hat{\boldsymbol{x}}} \overline{\tilde{\varphi}} \cdot \nabla_{\hat{\boldsymbol{x}}}  \hat{z}-\kappa^2\overline{\tilde{\varphi}} \hat{z}\right] \mathrm{~d} \boldsymbol{x} \\
		& =	\int_{\Omega_+^{\mathrm{PML}}} (\Delta_{\hat{\boldsymbol{x}}} \hat{z} + \kappa^2 \hat{z}) \overline{\tilde{\varphi}}\mathrm{~d} \boldsymbol{x} + \int_{\Gamma_+} \partial_{x_2}{\hat{z}} \overline{\tilde{\varphi}} \mathrm{~d} s\\
		&=\int_{\Gamma_+} \partial_{x_2}{\hat{z}}  \overline{\tilde{\varphi}} \mathrm{~d} s=\int_{\Gamma_+} \mathscr{T}_+^{\mathrm{PML}} p  \overline{\varphi} \mathrm{~d} s.
	\end{align*}
	Similarly, using the first Betti's formula and the extension yields
	\begin{align*}
		\int_{\Gamma_-} \boldsymbol{u} \cdot \mathscr{B} \overline{\tilde{\boldsymbol{\psi}}} \mathrm{~d} s & =\int_{\Gamma_-} \hat{\boldsymbol{w}} \cdot \mathscr{B} \overline{\tilde{\boldsymbol{\psi}}} \mathrm{~d} s \\
		&= -\int_{\Omega_-^{\mathrm{PML}}}\left[\lambda(\nabla_{\hat{\boldsymbol{x}}} \cdot \overline{\tilde{\boldsymbol{\psi}}})(\nabla_{\hat{\boldsymbol{x}}} \cdot \hat{\boldsymbol{w}})+\frac{\mu}{2}\left(\nabla_{\hat{\boldsymbol{x}}} \overline{\tilde{\boldsymbol{\psi}}}+\nabla_{\hat{\boldsymbol{x}}}  \overline{\tilde{\boldsymbol{\psi}}}^{\top}\right):\left(\nabla_{\hat{\boldsymbol{x}}} \hat{\boldsymbol{w}}+\nabla_{\hat{\boldsymbol{x}}} \hat{\boldsymbol{w}}^{\top}\right)-\omega^2\rho \overline{\tilde{\boldsymbol{\psi}}} \cdot \hat{\boldsymbol{w}}\right] \mathrm{d} \boldsymbol{x} \\
		& =\int_{\Omega_-^{\mathrm{PML}}}\left(\Delta^*_{\hat{\boldsymbol{x}}}\hat{\boldsymbol{w}}+\omega^2\rho \hat{\boldsymbol{w}}\right) \cdot \overline{\tilde{\boldsymbol{\psi}}} \mathrm{~d} \boldsymbol{x}+\int_{\Gamma_-} \mathscr{B} \hat{\boldsymbol{w}} \cdot \overline{\tilde{\boldsymbol{\psi}}} \mathrm{~d} s \\
		& =\int_{\Gamma_-} \mathscr{B} \hat{\boldsymbol{w}} \cdot \overline{\tilde{\boldsymbol{\psi}}} \mathrm{~d} s=\int_{\Gamma_-} \mathscr{T}_-^{\mathrm{PML}} \boldsymbol{u} \cdot \overline{\tilde{\boldsymbol{\psi}}} \mathrm{~d} s,
	\end{align*}
	which completes the proof.
\end{proof}
\begin{lemma}\label{lemma6}
	For any $\boldsymbol{V}=(\varphi,\boldsymbol{\psi}) \in \mathscr{H}^1_\mathrm{qp}\left(\Omega\right)$, let $\tilde{\boldsymbol{V}} \in \mathscr{H}_{0,\mathrm{qp}}^1(\mathrm{D})$ be the extension of $\boldsymbol{V}$ which satisfies (\ref{extenstion}). Then there holds
	\begin{align*}
		& \|\nabla \tilde{\varphi}\|_{L^2\left(\Omega_+^{\mathrm{PML}}\right)} \leq \gamma_1 C_1\|\varphi\|_{H^1(\Omega_+)},\\
		& \|\nabla \tilde{\boldsymbol{\psi}}\|_{F\left(\Omega_-^{\mathrm{PML}}\right)} \leq \gamma_1 C_2\|\boldsymbol{\psi}\|_{H^1(\Omega_-)^2},
	\end{align*}
	where $C_1>0$ and $C_2>0$ are 
	constants, and the Frobenius norm of the Jacobian matrix $\nabla \boldsymbol{\psi}$ is defined by
	\begin{align*}
		\|\nabla \boldsymbol{\psi}\|_{F(R)}=\left(\sum_{j=1}^2 \int_R\left|\nabla \psi_j\right|^2 \mathrm{d} \boldsymbol{x}\right)^{1 / 2}, \quad \text { for any subdomain } R \subset \mathbb{R}^2. 
	\end{align*}
\end{lemma}
\begin{proof}
	The proof follows a methodology analogous to that employed in \cite[Lemma 4.3]{Chen03} and \cite[Lemma 4.3]{Jiang17-3}. For brevity, the corresponding detailed derivations are omitted.
\end{proof}

For notational simplicity, we write $\tilde{\boldsymbol{V}}$ as  $\boldsymbol{V}$ in the sequel, unless stated otherwise. The following lemma gives an error representation formula, which plays a key role in the analysis of the a posteriori estimate.
\begin{lemma}\label{errorformula}
	For any $\boldsymbol{V}=(\varphi,\boldsymbol{\psi}) \in \mathscr{H}^1_\mathrm{qp}\left(\Omega\right)$, which can be extended to be a function in $\mathscr{H}_{0,\mathrm{qp}}^1\left(\mathrm{D}\right)$. Based on (\ref{extenstion}), and $\boldsymbol{V}_h=(\varphi_h,\boldsymbol{\psi}_h) \in \stackrel{\circ}{\mathscr{H}_h}(\mathrm{D})$, there holds
	\begin{align}\label{lemm6}
		A\left(\boldsymbol{U}-\hat{\boldsymbol{U}}_{h}, \boldsymbol{V}\right)=&~ L(\boldsymbol{V}-\boldsymbol{V}_h)
		-B_\mathrm{D}\left(\hat{\boldsymbol{U}}_h,\boldsymbol{V}-\boldsymbol{V}_h\right) \notag \\
		&+ \int_{\Gamma_+}\left(\mathscr{T}_+-\mathscr{T}_+^{\mathrm{PML}}\right)\hat{p}^{\mathrm{sc}}_h \overline{\varphi} \mathrm{~d} s
		+\int_{\Gamma_-}\left(\mathscr{T}_--\mathscr{T}_-^{\mathrm{PML}}\right)\hat{\boldsymbol{u}}_h \cdot \overline{\boldsymbol{\psi}} \mathrm{~d} s.
	\end{align}
\end{lemma}
\begin{proof}
	By (\ref{DtN varproblem}), (\ref{DtN varleft}), (\ref{PML varproblem}) and (\ref{PML varleft}), one can get
	\begin{align}
		A\left(\boldsymbol{U}-\hat{\boldsymbol{U}}_{h}, \boldsymbol{V}\right)=&~A\left(\boldsymbol{U}-\hat{\boldsymbol{U}}, \boldsymbol{V}\right)
		+ A\left(\hat{\boldsymbol{U}}-\hat{\boldsymbol{U}}_h, \boldsymbol{V}\right) \notag\\
		=&~A\left(\boldsymbol{U}-\hat{\boldsymbol{U}}, \boldsymbol{V}\right)
		+ A\left(\hat{\boldsymbol{U}}-\hat{\boldsymbol{U}}_h, \boldsymbol{V}\right) 
		-A^{\mathrm{PML}}\left(\hat{\boldsymbol{U}}-\hat{\boldsymbol{U}}_h, \boldsymbol{V}\right)
		+A^{\mathrm{PML}}\left(\hat{\boldsymbol{U}}-\hat{\boldsymbol{U}}_h, \boldsymbol{V}\right) \notag\\
		=&~A\left(\boldsymbol{U},\boldsymbol{V}\right)-A\left(\hat{\boldsymbol{U}},\boldsymbol{V}\right)+A^{\mathrm{PML}}\left(\hat{\boldsymbol{U}}-\hat{\boldsymbol{U}}_h, \boldsymbol{V}\right) \notag\\
		&-\int_{\Gamma_+}\left(\mathscr{T}_+-\mathscr{T}_+^{\mathrm{PML}}\right)\left(\hat{p}^{\mathrm{sc}}-\hat{p}^{\mathrm{sc}}_h\right) \overline{\varphi} \mathrm{~d}s -\int_{\Gamma_-}\left(\mathscr{T}_--\mathscr{T}_-^{\mathrm{PML}}\right)\left(\hat{\boldsymbol{u}}-\hat{\boldsymbol{u}}_h\right) \cdot \overline{\boldsymbol{\psi}} \mathrm{~d} s\notag \\
		=&~A\left(\boldsymbol{U},\boldsymbol{V}\right)-A\left(\hat{\boldsymbol{U}},\boldsymbol{V}\right)+A^{\mathrm{PML}}\left(\hat{\boldsymbol{U}},\boldsymbol{V}\right)-A^{\mathrm{PML}}\left(\hat{\boldsymbol{U}}, \boldsymbol{V}\right)
		+A^{\mathrm{PML}}\left(\hat{\boldsymbol{U}}-\hat{\boldsymbol{U}}_h, \boldsymbol{V}\right) \notag\\
		&-\int_{\Gamma_+}\left(\mathscr{T}_+-\mathscr{T}_+^{\mathrm{PML}}\right)\left(\hat{p}^{\mathrm{sc}}-\hat{p}^{\mathrm{sc}}_h\right) \overline{\varphi} \mathrm{~d} s
		-\int_{\Gamma_-}\left(\mathscr{T}_--\mathscr{T}_-^{\mathrm{PML}}\right)\left(\hat{\boldsymbol{u}}-\hat{\boldsymbol{u}}_h\right) \cdot \overline{\boldsymbol{\psi}} \mathrm{~d} s \notag\\
		=&\int_{\Gamma_+}\left(\mathscr{T}_+-\mathscr{T}_+^{\mathrm{PML}}\right)\hat{p}^{\mathrm{sc}} \overline{\varphi} \mathrm{~d} s
		+\int_{\Gamma_-}\left(\mathscr{T}_--\mathscr{T}_-^{\mathrm{PML}}\right)\hat{\boldsymbol{u}} \cdot \overline{\boldsymbol{\psi}} \mathrm{~d} s +A^{\mathrm{PML}}\left(\hat{\boldsymbol{U}}-\hat{\boldsymbol{U}}_h, \boldsymbol{V}\right) \notag\\
		&-\int_{\Gamma_+}\left(\mathscr{T}_+-\mathscr{T}_+^{\mathrm{PML}}\right)\left(\hat{p}^{\mathrm{sc}}-\hat{p}^{\mathrm{sc}}_h\right) \overline{\varphi} \mathrm{~d} s
		-\int_{\Gamma_-}\left(\mathscr{T}_--\mathscr{T}_-^{\mathrm{PML}}\right)\left(\hat{\boldsymbol{u}}-\hat{\boldsymbol{u}}_h\right) \cdot \overline{\boldsymbol{\psi}} \mathrm{~d} s \notag\\
		=&\int_{\Gamma_+}\left(\mathscr{T}_+-\mathscr{T}_+^{\mathrm{PML}}\right)\hat{p}^{\mathrm{sc}}_h \overline{\varphi} \mathrm{~d} s
		+\int_{\Gamma_-}\left(\mathscr{T}_--\mathscr{T}_-^{\mathrm{PML}}\right)\hat{\boldsymbol{u}}_h \cdot \overline{\boldsymbol{\psi}} \mathrm{~d} s +A^{\mathrm{PML}}\left(\hat{\boldsymbol{U}}-\hat{\boldsymbol{U}}_h, \boldsymbol{V}\right).\label{4.2.1}
	\end{align}
	Combining equation (\ref{PML varleft}) and Lemma \ref{lemma5} yields
	\begin{align*}
		A^{\mathrm{PML}}\left(\hat{\boldsymbol{U}}-\hat{\boldsymbol{U}}_h,\boldsymbol{V}\right)
		=&~B_{\Omega}\left(\hat{\boldsymbol{U}}-\hat{\boldsymbol{U}}_h, \boldsymbol{V}\right) -\int_{\Gamma_+}\mathscr{T}_+^{\mathrm{PML}}\left(\hat{p}^{\mathrm{sc}}-\hat{p}^{\mathrm{sc}}_h\right) \overline{\varphi} \mathrm{~d} s
		-\int_{\Gamma_-}\mathscr{T}_-^{\mathrm{PML}}\left(\hat{\boldsymbol{u}}-\hat{\boldsymbol{u}}_h\right) \cdot \overline{\boldsymbol{\psi}} \mathrm{~d} s\\
		=&~B_{\Omega}\left(\hat{\boldsymbol{U}}-\hat{\boldsymbol{U}}_h, \boldsymbol{V}\right) 
		- \int_{\Gamma_+}\left(\hat{p}^{\mathrm{sc}}-\hat{p}^{\mathrm{sc}}_h\right) \partial_{x_2} \overline{\varphi} \mathrm{~d} s
		- \int_{\Gamma_-} \left(\hat{\boldsymbol{u}}-\hat{\boldsymbol{u}}_h\right) \cdot \mathscr{B} \overline{\boldsymbol{\psi}} \mathrm{~d} s.
	\end{align*}
	Due to $\mathcal{L}_1 \overline{\varphi}=0$ in $\Omega^{\mathrm{PML}}_+$ and $\mathcal{L}_2 \overline{\boldsymbol{\psi}}=0$ in $\Omega^{\mathrm{PML}}_-$, it follows from the Green's formula and the first Betti's formula that
	\begin{align*}
		0 =& -\int_{\Omega^{\mathrm{PML}}_+}	\mathcal{L}_1\overline{\varphi}\left(\hat{p}^{\mathrm{sc}}-\hat{p}^{\mathrm{sc}}_h\right) \mathrm{~d} \boldsymbol{x} \notag\\
		=& \int_{\Omega^{\mathrm{PML}}_+}\left[\mathbb{A}\nabla \overline{\varphi} \cdot \nabla  \left(\hat{p}^{\mathrm{sc}}-\hat{p}^{\mathrm{sc}}_h\right)   
		- \kappa^2 s(x_2) \overline{\varphi}\left(\hat{p}^{\mathrm{sc}}-\hat{p}^{\mathrm{sc}}_h\right)\right] \mathrm{~d} \boldsymbol{x} 
		- \int_{\Gamma_+^\mathrm{PML}}\left(\hat{p}^{\mathrm{sc}}-\hat{p}^{\mathrm{sc}}_h\right) \partial_{\hat{x}_2} \overline{\varphi} \mathrm{~d} s
		+ \int_{\Gamma_+}\left(\hat{p}^{\mathrm{sc}}-\hat{p}^{\mathrm{sc}}_h\right) \partial_{x_2} \overline{\varphi} \mathrm{~d} s
	\end{align*}
	and 
	\begin{align*}
		0=&-\int_{\Omega^{\mathrm{PML}}_-}\mathcal{L}_2\overline{\boldsymbol{\psi}}\cdot\left(\hat{\boldsymbol{u}}-\hat{\boldsymbol{u}}_h\right)   \mathrm{~d} \boldsymbol{x} \notag\\
		= &\int_{\Omega^{\mathrm{PML}}_-}\left[\mathcal{S}_{\lambda, \mu}(\hat{\boldsymbol{u}}-\hat{\boldsymbol{u}}_h, \boldsymbol{\psi}) - \omega^2 s(x_2) \rho \left(\hat{\boldsymbol{u}}-\hat{\boldsymbol{u}}_h\right) \cdot \overline{\boldsymbol{\psi}}\right] \mathrm{~d} \boldsymbol{x}
		- \int_{\Gamma_-^\mathrm{PML}} \left(\hat{\boldsymbol{u}}-\hat{\boldsymbol{u}}_h\right) \cdot \mathscr{B}_{\hat{\boldsymbol{x}}} \overline{\boldsymbol{\psi}} \mathrm{~d} s
		+ \int_{\Gamma_-} \left(\hat{\boldsymbol{u}}-\hat{\boldsymbol{u}}_h\right) \cdot \mathscr{B} \overline{\boldsymbol{\psi}} \mathrm{~d} s.
	\end{align*}
	Furthermore, it follows from (\ref{trun PMLvarproblem}), (\ref{trun APPLvarproblem}) and the above two equalities that
	\begin{align}\label{4.2.2}
		A^{\mathrm{PML}}\left(\hat{\boldsymbol{U}}-\hat{\boldsymbol{U}}_h,\boldsymbol{V}\right)
		=&~ B_{\Omega}\left(\hat{\boldsymbol{U}}-\hat{\boldsymbol{U}}_h, \boldsymbol{V}\right) - \int_{\Gamma_+}\left(\hat{p}^{\mathrm{sc}}-\hat{p}^{\mathrm{sc}}_h\right) \partial_{x_2} \overline{\varphi} \mathrm{~d} s
		- \int_{\Gamma_-} \left(\hat{\boldsymbol{u}}-\hat{\boldsymbol{u}}_h\right) \cdot \mathscr{B} \overline{\boldsymbol{\psi}} \mathrm{~d} s \notag\\
		=&~B_\mathrm{D}\left(\hat{\boldsymbol{U}}-\hat{\boldsymbol{U}}_h, \boldsymbol{V}\right)
		- \int_{\Gamma_+^\mathrm{PML}}\left(\hat{p}^{\mathrm{sc}}-\hat{p}^{\mathrm{sc}}_h\right) \partial_{\hat{x}_2} \overline{\varphi} \mathrm{~d} s
		- \int_{\Gamma_-^\mathrm{PML}} \left(\hat{\boldsymbol{u}}-\hat{\boldsymbol{u}}_h\right) \cdot \mathscr{B}_{\hat{\boldsymbol{x}}} \overline{\boldsymbol{\psi}} \mathrm{~d} s\notag\\
		=&~ B_\mathrm{D}\left(\hat{\boldsymbol{U}}, \boldsymbol{V}\right) -B_\mathrm{D}\left(\hat{\boldsymbol{U}}_h, \boldsymbol{V}\right)
		+ B_\mathrm{D}\left(\hat{\boldsymbol{U}}_h, \boldsymbol{V}_h\right)
		- B_\mathrm{D}\left(\hat{\boldsymbol{U}}_h, \boldsymbol{V}_h\right)
		\notag\\
		&-\int_{\Gamma_+^\mathrm{PML}}\left(\hat{p}^{\mathrm{sc}}-\hat{p}^{\mathrm{sc}}_h\right) \partial_{\hat{x}_2} \overline{\varphi} \mathrm{~d} s
		- \int_{\Gamma_-^\mathrm{PML}} \left(\hat{\boldsymbol{u}}-\hat{\boldsymbol{u}}_h\right) \cdot \mathscr{B}_{\hat{\boldsymbol{x}}} \overline{\boldsymbol{\psi}} \mathrm{~d} s \notag\\
		=& ~L(\boldsymbol{V}) - B_\mathrm{D}\left(\hat{\boldsymbol{U}}_h, \boldsymbol{V}-\boldsymbol{V}_h\right)- L(\boldsymbol{V}_h)
		\notag\\
		&-\int_{\Gamma_+^\mathrm{PML}}\left(\hat{p}^{\mathrm{sc}}-\hat{p}^{\mathrm{sc}}_h\right) \partial_{\hat{x}_2} \overline{\varphi} \mathrm{~d} s
		- \int_{\Gamma_-^\mathrm{PML}} \left(\hat{\boldsymbol{u}}-\hat{\boldsymbol{u}}_h\right) \cdot \mathscr{B}_{\hat{\boldsymbol{x}}} \overline{\boldsymbol{\psi}} \mathrm{~d} s \notag\\
		=&~ L(\boldsymbol{V}-\boldsymbol{V}_h)
		-B_\mathrm{D}\left(\hat{\boldsymbol{U}}_h,\boldsymbol{V}-\boldsymbol{V}_h\right)
		-\int_{\Gamma_+^\mathrm{PML}}\left(\hat{p}^{\mathrm{sc}}-\hat{p}^{\mathrm{sc}}_h\right) \partial_{\hat{x}_2} \overline{\varphi} \mathrm{~d} s
		- \int_{\Gamma_-^\mathrm{PML}} \left(\hat{\boldsymbol{u}}-\hat{\boldsymbol{u}}_h\right) \cdot \mathscr{B}_{\hat{\boldsymbol{x}}} \overline{\boldsymbol{\psi}} \mathrm{~d} s.
	\end{align}
	Substituting (\ref{4.2.2}) into (\ref{4.2.1}) gives
	\begin{align*}
		A\left(\boldsymbol{U}-\hat{\boldsymbol{U}}_{h}, \boldsymbol{V}\right)=&~ L(\boldsymbol{V}-\boldsymbol{V}_h)
		-B_\mathrm{D}\left(\hat{\boldsymbol{U}}_h,\boldsymbol{V}-\boldsymbol{V}_h\right) \\
		&+ \int_{\Gamma_+}\left(\mathscr{T}_+-\mathscr{T}_+^{\mathrm{PML}}\right)\hat{p}^{\mathrm{sc}}_h \overline{\varphi} \mathrm{~d} s
		+\int_{\Gamma_-}\left(\mathscr{T}_s-\mathscr{T}_s^{\mathrm{PML}}\right)\hat{\boldsymbol{u}}_h \cdot \overline{\boldsymbol{\psi}} \mathrm{~d} s,
	\end{align*}
	where we have used $\hat{p}^{\mathrm{sc}}=\hat{p}^{\mathrm{sc}}_h=0$ on $\Gamma_+^{\mathrm{PML}}$ and $\hat{\boldsymbol{u}}=\hat{\boldsymbol{u}}_h=0$ on $\Gamma_-^{\mathrm{PML}}$. Thus the proof is completed.
\end{proof}

\subsection{A posteriori error analysis}
In this subsection, we introduce the definition of a posteriori error indicators. For any $T \in \mathscr{M}_h$, define $h_T$ as its diameter. Let $\partial T$ be the set of all the edges of $T$ that do not lie on $\Gamma_+^{\mathrm{PML}}$ and $\Gamma_-^{\mathrm{PML}}$. For any $e \in \partial T$, define $h_e$ as its length. Let
\begin{align*}
	\mathcal{R}_+ \hat{p}^{\mathrm{sc}}_h,& :=\mathcal{L}_1 \hat{p}^{\mathrm{sc}}_h, \quad \forall~T \in \mathscr{M}_{h, +}, \\
	\mathcal{R}_- \hat{\boldsymbol{u}}_h& :=\mathcal{L}_2 \hat{\boldsymbol{u}}_h, \quad \forall~T \in \mathscr{M}_{h, -}.
\end{align*}
For any $T_1^+, T_2^+ \in \mathscr{M}_{h, +}$, the jump residual $J_{e, +}$ is given by
\begin{align*}
	J_{e, +}:=
	\begin{cases}
		- (\left.\mathbb{A} \nabla \hat{p}^{\mathrm{sc}}_h\right|_{T_1^+} \cdot \boldsymbol{\nu}_1 + \left.\mathbb{A} \nabla \hat{p}^{\mathrm{sc}}_h\right|_{T_2^+} \cdot \boldsymbol{\nu}_2), & e \in \partial {T_1^+} \cap \partial {T_2^+},\\ 2\left(\partial_{\boldsymbol{n}}\left(p^\mathrm{in}+\hat{p}^{\mathrm{sc}}_h\right)-\rho_f \omega^2 \hat{\boldsymbol{u}}_h \cdot \boldsymbol{n}\right), & e \in \Gamma.
	\end{cases}
\end{align*}
where $\boldsymbol{\nu}_j$ is the unit outward normal vector on the boundary of $T_j^+$ ($j=1,2$). For any $T_1^-, T_2^- \in \mathscr{M}_{h, -}$, the jump residual $\boldsymbol{J}_{e, -}$ is given by
\begin{align*}
	\boldsymbol{J}_{e, -}:=
	\begin{cases}
		-\left(\left.\boldsymbol{T} \hat{\boldsymbol{u}}_h\right|_{T_1^-}+\left.\boldsymbol{T} \hat{\boldsymbol{u}}_h\right|_{T_1^-}\right), & e \in \partial {T_1^-} \cap \partial {T_2^-},\\ -2\left(\left(p^\mathrm{in}+\hat{p}^{\mathrm{sc}}_h\right) \boldsymbol{n}+\boldsymbol{T} \hat{\boldsymbol{u}}_h\right), & e \in \Gamma.
	\end{cases}
\end{align*}
Next we consider the jump residual on the left line segment $\Gamma_{\text {left}}$ and right line segment $\Gamma_{\text {right}}$, where
\begin{align*}
	\Gamma_{\text {left}} &=\left\{\left(x_1, x_2\right): x_1 =0, h_2-\delta_2<x_2<h_1+\delta_1\right\},\\
	\Gamma_{\text {right}} &=\left\{\left(x_1, x_2\right): x_1=\Lambda, h_2-\delta_2<x_2<h_1+\delta_1\right\}.
\end{align*}
For any $e \in\Gamma_{\text {left}} \cap \partial T_1$ for some $T_1 \in \mathscr{M}_h$ and its corresponding edge $e^{\prime} \in\Gamma_{\text {right}} \cap \partial T_2$  for some $T_2 \in \mathscr{M}_h$, denote the jump residual by
\begin{align*}
	J_{e, +}&:= - (\left.\mathbb{A} \nabla \hat{p}^{\mathrm{sc}}_h\right|_{T_1} \cdot \boldsymbol{\nu}_1 + e^{-\mathrm{i}\alpha \Lambda} \left.\mathbb{A} \nabla \hat{p}^{\mathrm{sc}}_h\right|_{T_2} \cdot \boldsymbol{\nu}_2),\\
	J_{e^{\prime}, +} &:= - (\left. e^{\mathrm{i}\alpha \Lambda} \mathbb{A} \nabla \hat{p}^{\mathrm{sc}}_h\right|_{T_1} \cdot \boldsymbol{\nu}_1 + \left.\mathbb{A} \nabla \hat{p}^{\mathrm{sc}}_h\right|_{T_2} \cdot \boldsymbol{\nu}_2),
\end{align*}
and
\begin{align*}
	\boldsymbol{J}_{e, -}&:=  -\left(\left.\boldsymbol{T} \hat{\boldsymbol{u}}_h\right|_{T_1}+\left.e^{-\mathrm{i} \alpha \Lambda} \boldsymbol{T} \hat{\boldsymbol{u}}_h\right|_{T_2}\right), \\
	\boldsymbol{J}_{e^{\prime}, -} &:=  -\left(\left.e^{\mathrm{i} \alpha \Lambda} \boldsymbol{T} \hat{\boldsymbol{u}}_h\right|_{T_1}+\left.\boldsymbol{T} \hat{\boldsymbol{u}}_h\right|_{T_2}\right).
\end{align*}
For $T \in \mathscr{M}_{h, +}$ and $T \in \mathscr{M}_{h, -}$, denote the local error estimators by $\eta_{T, +}$ and $\eta_{T, -}$, where
\begin{align}
	& \eta_{T, +}=h_T\left\|\mathcal{R}_+  \hat{p}^{\mathrm{sc}}_h\right\|_{L^2(T)}+\left(\frac{1}{2} \sum_{e \in \partial T} h_e\left\|J_{e, +}\right\|_{L^2(e)}^2\right)^{1 / 2}, \label{eta1}\\
	& \eta_{T, -}=h_T\left\|\mathcal{R}_- \hat{\boldsymbol{u}}_h\right\|_{L^2(T)^2}+\left(\frac{1}{2} \sum_{e \in \partial T} h_e\left\|\boldsymbol{J}_{e, -}\right\|_{L^2(e)^2}^2\right)^{1 / 2}.\label{eta2}
\end{align}

Now we are ready to give the main result.
\begin{theorem}\label{theorem main}
	There exists a positive constant $C$ such that the following a posteriori error estimate holds
	\begin{align*}
		\|\boldsymbol{U}-\hat{\boldsymbol{U}}_h\|_{\mathscr{H}^1_\mathrm{qp}(\Omega)}
		\leq&~ \max\left\{C(1+\gamma_1C_1),C(1+\gamma_1C_2)\right\}\left(\sum_{T \in \mathscr{M}_{h, +}} \eta_{T, +}^2+\sum_{T \in \mathscr{M}_{h, -}} \eta_{T, -}^2\right)^{1 / 2}\\
		&+\gamma_1 F_1\|\hat{p}^{\mathrm{sc}}_h \|_{L^2(\Gamma_+)} + \gamma_1 F_2\|\hat{\boldsymbol{u}}_h \|_{L^2(\Gamma_-)^2},
	\end{align*}
	where $\gamma_1$, $F_j$ and  $C_j$ $(j=1,2)$  are constants defined in Lemmas \ref{trace}, \ref{PML-DtNerror} and \ref{lemma6}, respectively.
\end{theorem}
\begin{proof}
	For brevity, we can rewrite the formula (\ref{lemm6}) as
	\begin{align}\label{4.3.1}
		A\left(\boldsymbol{U}-\hat{\boldsymbol{U}}_{h}, \boldsymbol{V}\right)
		=&~ {J}_1 + {J}_2,
	\end{align}
	where
	\begin{align*}
		{J}_1 &=~ L(\boldsymbol{V}-\boldsymbol{V}_h)
		-B_\mathrm{D}\left(\hat{\boldsymbol{U}}_h,\boldsymbol{V}-\boldsymbol{V}_h\right),\\
		{J}_2 & =~ \int_{\Gamma_+}\left(\mathscr{T}_+-\mathscr{T}_+^{\mathrm{PML}}\right)\hat{p}^{\mathrm{sc}}_h \overline{\varphi} \mathrm{~d} s
		+\int_{\Gamma_-}\left(\mathscr{T}_s-\mathscr{T}_s^{\mathrm{PML}}\right)\hat{\boldsymbol{u}}_h \cdot \overline{\boldsymbol{\psi}} \mathrm{~d} s.
	\end{align*}
	Take $\varphi_h=\Pi_h^+ \varphi \in H_{0, \mathrm{qp}}^1\left(\mathrm{D}_+\right)$ and $\boldsymbol{\psi}_h=\Pi_h^-  \boldsymbol{\psi} \in H_{0, \mathrm{qp}}^1\left(\mathrm{D}_-\right)^2$, where $\Pi_h^+$ and $\Pi_h^-$ are Scott-Zhang interpolation operators (cf. \cite{Scott90}) satisfying the following estimates
	\begin{align}
		\left\|\varphi-\Pi_h^+ \varphi\right\|_{L^2(T)} \leq C h_T\|\nabla \varphi\|_{L^2(\tilde{T})},& \quad\left\|\varphi-\Pi_h^+ \varphi\right\|_{L^2(e)} \leq C h_e^{1 / 2}\|\nabla \varphi\|_{L^2(\tilde{e})},\label{inter1}\\
		\left\|\boldsymbol{\psi}-\Pi_h^- \boldsymbol{\psi}\right\|_{L^2(T)^2} \leq C h_T\|\nabla \boldsymbol{\psi}\|_{F(\tilde{T})},& \quad\left\|\boldsymbol{\psi}-\Pi_h^- \boldsymbol{\psi}\right\|_{L^2(e)^2} \leq C h_e^{1 / 2}\|\nabla \boldsymbol{\psi}\|_{F(\tilde{e})}.\label{inter2}
	\end{align}
	Here, $\tilde{T}$ and $\tilde{e}$ are the union of all elements in $\mathscr{M}_{h,+} \cup \mathscr{M}_{h, -}$, which have nonempty intersection with $T$ and  $e$, respectively. By the definition of $B_\mathrm{D}(\cdot, \cdot)$ and $L(\cdot)$ in (\ref{trun APPLvarproblem}), one has
	\begin{align*}
		{J}_1 &= \sum_{T \in \mathscr{M}_{h, +}}  \left(\int_T\left(-\mathbb{A}\nabla  \hat{p}^{\mathrm{sc}}_h \cdot \nabla\left(\overline{\varphi}-\Pi_h^+\overline{\varphi}\right)+\kappa^2 s(x_2) \hat{p}^{\mathrm{sc}}_h\left(\overline{\varphi}-\Pi_h^+\overline{\varphi} \right)\right) \mathrm{~d} \boldsymbol{x}\right. \\
		& \quad\qquad\qquad \left.+\sum_{e \in \partial T \cap \Gamma} \int_e\left(\partial_{\boldsymbol{n}} p^\mathrm{in}-\rho_f \omega^2 \hat{\boldsymbol{u}}_h \cdot \boldsymbol{n}\right)\left(\overline{\varphi}-\Pi_h^+\overline{\varphi}\right) \mathrm{~d} s\right) \\
		& \quad + \sum_{T \in \mathscr{M}_{h, -}}
		\left(\int_T-\mathcal{S}_{\lambda,\mu}\left(\hat{\boldsymbol{u}}_h, \boldsymbol{\psi}-\Pi_h^-\boldsymbol{\psi}\right)+ \omega^2 s(x_2) \rho  \hat{\boldsymbol{u}}_h
		\cdot\left(\overline{\boldsymbol{\psi}}-\Pi_h^-\overline{\boldsymbol{\psi}}\right) \mathrm{~d} \boldsymbol{x}\right. \\
		& \quad \qquad\qquad \left.+\sum_{e \in \partial T \cap \Gamma} \int_e-\left(p^\mathrm{in}+ \hat{p}^{\mathrm{sc}}_h\right) \boldsymbol{n} \cdot\left(\overline{\boldsymbol{\psi}}-\Pi_h^-\overline{\boldsymbol{\psi}}\right) \mathrm{~d} s\right).
	\end{align*}
	Applying the integration by parts yields
	\begin{align*}
		{J}_1 &= \sum_{T \in \mathscr{M}_{h, +}}  \left(\int_T \mathcal{L}_1 \hat{p}^{\mathrm{sc}}_h\left(\overline{\varphi}-\Pi_h^+\overline{\varphi} \right) \mathrm{~d} \boldsymbol{x}\right.- \sum_{e \in \partial T}\int_e \mathbb{A} \nabla \hat{p}^{\mathrm{sc}}_h \cdot \boldsymbol{\nu}(\overline{\varphi}-\Pi_h^+\overline{\varphi}) \mathrm{~d} \boldsymbol{x} \\
		& \quad\qquad\qquad \left.+\sum_{e \in \partial T \cap \Gamma}  \int_e\left(\partial_{\boldsymbol{n}} p^\mathrm{in}-\rho_f \omega^2 \hat{\boldsymbol{u}}_h \cdot \boldsymbol{n}\right)\left(\overline{\varphi}-\Pi_h^+\overline{\varphi}\right) \mathrm{~d} s\right) \\
		&\quad + \sum_{T \in \mathscr{M}_{h, -}}\left(\int_T\mathcal{L}_2 \hat{\boldsymbol{u}}_h\cdot\left(\overline{\boldsymbol{\psi}}-\Pi_h^-\overline{\boldsymbol{\psi}}\right) \mathrm{~d} \boldsymbol{x}\right.-\sum_{e \in \partial T}\int_e \boldsymbol{T} \hat{\boldsymbol{u}}_h\left(\overline{\boldsymbol{\psi}}-\Pi_h^-\overline{\boldsymbol{\psi}}\right) \mathrm{~d} \boldsymbol{x} \\
		& \quad\qquad\qquad \left.+\sum_{e \in \partial T \cap \Gamma} \int_e-\left(p^\mathrm{in}+ \hat{p}^{\mathrm{sc}}_h\right) \boldsymbol{n} \cdot\left(\overline{\boldsymbol{\psi}}-\Pi_h^-\overline{\boldsymbol{\psi}}\right) \mathrm{~d} s\right)\\
		& =\sum_{T \in \mathscr{M}_{h, +}}
		\left(\int_T \mathcal{R}_+\hat{p}^{\mathrm{sc}}_h\left(\overline{\varphi}-\Pi_h^+\overline{\varphi} \right) \mathrm{~d} \boldsymbol{x}+\sum_{e \in \partial T} \frac{1}{2} \int_e J_{e, +}\left(\overline{\varphi}-\Pi_h^+\overline{\varphi} \right) \mathrm{~d} s\right) \\
		&\quad +\sum_{T \in \mathscr{M}_{h, -}}\left(\int_T \mathcal{R}_- \hat{\boldsymbol{u}}_h \cdot\left(\overline{\boldsymbol{\psi}}-\Pi_h^-\overline{\boldsymbol{\psi}}\right) \mathrm{~d} \boldsymbol{x}+\sum_{e \in \partial T} \frac{1}{2} \int_e \boldsymbol{J}_{e, -} \cdot\left(\overline{\boldsymbol{\psi}}-\Pi_h^-\overline{\boldsymbol{\psi}}\right) \mathrm{~d} s\right).
	\end{align*}
	It follows from (\ref{inter1}) and (\ref{inter2}) that
	\begin{align*}
		|{J}_1| \leq& \sum_{T \in \mathscr{M}_{h, +}}\left(Ch_T\left\|\mathcal{R}_+\hat{p}^{\mathrm{sc}}_h\right\|_{L^2(T)}\|\nabla \varphi\|_{L^2(\tilde{T})}+\sum_{e \in \partial T} \frac{1}{2} Ch_e^{1 / 2}\left\|J_{e, +}\right\|_{L^2(e)}\|\nabla \varphi\|_{L^2(\tilde{e})}\right) \\
		&+  \sum_{T \in \mathscr{M}_{h, -}}\left(Ch_T\left\|\mathcal{R}_- \hat{\boldsymbol{u}}_h\right\|_{L^2(T)^2}\|\nabla \boldsymbol{\psi}\|_{F(\tilde{T})}+\sum_{e \in \partial T} \frac{1}{2} C h_e^{1 / 2}\left\|\boldsymbol{J}_{e,-}\right\|_{L^2(e)^2}\|\nabla \boldsymbol{\psi}\|_{F(\tilde{e})}\right)\\
		\leq&~C\sum_{T \in \mathscr{M}_{h,+}} \left(h_T\left\|\mathcal{R}_+ \hat{p}^{\mathrm{sc}}_h\right\|_{L^2(T)}+\left(\frac{1}{2} \sum_{e \in \partial T} h_e\left\|J_{e, +}\right\|_{L^2(e)}^2\right)^{1 / 2}\right)\|\nabla \varphi\|_{L^2(\tilde{T})}\\
		& +~ C\sum_{T \in \mathscr{M}_{h, -}} \left(h_T\left\|\mathcal{R}_- \hat{\boldsymbol{u}}_h\right\|_{L^2(T)^2}+\left(\frac{1}{2} \sum_{e \in \partial T} h_e\left\|\boldsymbol{J}_{e, -}\right\|_{L^2(e)^2}^2\right)^{1 / 2}\right)\|\nabla \boldsymbol{\psi}\|_{F(\tilde{T})} \\
		\leq&~ C\sum_{T \in \mathscr{M}_{h,+}} \eta_{T,+}\|\nabla \varphi\|_{L^2(\tilde{T})} 
		+ C\sum_{T \in \mathscr{M}_{h,-}} \eta_{T,-}\|\nabla \boldsymbol{\psi}\|_{F(\tilde{T})}.
	\end{align*}
	Using Lemma \ref{lemma6} and the Cauchy-Schwarz inequality gives
	\begin{align*}
		|{J}_1|
		\leq&~ C(1+\gamma_1C_1)\left(\sum_{T \in \mathscr{M}_{h, +}} \eta_{T, +}^2\right)^{1 / 2}\|\varphi\|_{H^1\left(\Omega_+\right)}
		+  C(1+\gamma_1C_2)\left(\sum_{T \in \mathscr{M}_{h, -}} \eta_{T, -}^2\right)^{1 / 2}\|\boldsymbol{\psi}\|_{H^1\left(\Omega_-\right)^2}\\
		\leq&~ \max\left\{C(1+\gamma_1C_1),C(1+\gamma_1C_2)\right\}\left(\sum_{T \in \mathscr{M}_{h, +}} \eta_{T, +}^2+\sum_{T \in \mathscr{M}_{h, -}} \eta_{T, -}^2\right)^{1 / 2}\|\boldsymbol{V}\|_{\mathscr{H}^1_\mathrm{qp}(\Omega)}
	\end{align*}
	By Lemmas \ref{trace} and \ref{PML-DtNerror}, we have
	\begin{align*}
		|J_2| &\leq F_1\|\hat{p}^{\mathrm{sc}}_h \|_{L^2(\Gamma_+)}\|\varphi\|_{L^2(\Gamma_+)} + F_2\|\hat{\boldsymbol{u}}_h \|_{L^2(\Gamma_-)^2}\|\boldsymbol{\psi}\|_{L^2(\Gamma_-)^2} \\
		& \leq \gamma_1 F_1\|\hat{p}^{\mathrm{sc}}_h \|_{L^2(\Gamma_+)}\|\varphi\|_{H^1(\Omega_+)} + \gamma_1 F_2\|\hat{\boldsymbol{u}}_h \|_{L^2(\Gamma_-)^2}\|\boldsymbol{\psi}\|_{H^1(\Omega_-)^2} \\
		& \leq \left(\gamma_1 F_1\|\hat{p}^{\mathrm{sc}}_h \|_{L^2(\Gamma_+)} + \gamma_1 F_2\|\hat{\boldsymbol{u}}_h \|_{L^2(\Gamma_-)^2}\right)\|\boldsymbol{V}\|_{\mathscr{H}^1_\mathrm{qp}(\Omega)}.
	\end{align*}
	Then the proof is completed by using the above estimates in (\ref{4.3.1}) and the inf-sup condition (\ref{dtn inf-sup}).
\end{proof}

\section{Implementation and Numerical Examples} \label{SECTION5}
This section presents the numerical implementation of the PML-AFEM algorithm by using the MATLAB toolbox, and validates the robustness of the proposed method through several numerical examples.

\subsection{Adaptive algorithm}
Based on Theorem \ref{theorem main}, we can obviously see that the a posteriori error estimate contains two components, i.e., the FE discretization error $\epsilon_{\mathrm{F}}$ and the PML approximation error $\epsilon_{\mathrm{P}}$, where
\begin{align}
	\epsilon_{\mathrm{F}} &=\left(\sum_{T \in \mathscr{M}_{h, +}} \eta_{T, +}^2+\sum_{T \in \mathscr{M}_{h, -}} \eta_{T, -}^2\right)^{1 / 2}, \label{5.1}\\
	\epsilon_{\mathrm{P}} &= F_1\|\hat{p}^{\mathrm{sc}}_h \|_{L^2(\Gamma_+)} +  F_2\|\hat{\boldsymbol{u}}_h \|_{L^2(\Gamma_-)^2}.\label{5.2}
\end{align}
According to (\ref{5.2}), we can select suitable PML parameters $\delta_j$ and $\sigma_j$ satisfying $F_1 \Lambda^{1 / 2} \leq 10^{-8}$ and $F_2 \Lambda^{1 / 2} \leq 10^{-8}$ to significantly reduce the PML error. In our numerical experiments, we always take $m=2$ and $\delta_1=\delta_2=\delta$. Algorithm \ref{algorithm} outlines the implementation steps of the adaptive $P_1$-FEM algorithm.

\begin{algorithm} 
	\caption{\enskip Adaptive PML-FEM algorithm for  the acoustic-elastic interaction in periodic structures} 
	\label{algorithm}
	\begin{algorithmic}[1]
		\State Set  the incident angle $\theta$ and the other parameters $\rho, \rho_f, \lambda, \mu, \omega$;
		\State Given the tolerance $\mathrm{tol} > 0$ and the threshold value $\tau \in (0,1)$;
		\State Choose $\delta$ and $\sigma_j$ such that $F_1 \Lambda^{1 / 2} \leq 10^{-8}$ and $F_2 \Lambda^{1 / 2} \leq 10^{-8}$;
		\State Construct an initial triangulation $\mathscr{M}_h$ over $\mathrm{D}$;
		\State Compute the numerical solution $\hat{\boldsymbol{U}}_h$ of the problem (\ref{trun APPLvarproblem}) and calculate error estimators;
		\While{$ \epsilon_\mathrm{F} > \mathrm{tol}$}
		\State Modify the mesh $\mathscr{M}_h$ based on the following strategy:
		\If {$\eta_{T, +}/\eta_{T, -}> \tau \displaystyle \max_{T \in \mathscr{M}_h} \left\{\eta_{T, +}/\eta_{T, -}\right\}$} 
		\State  Refine the element $\hat{T} \in \mathscr{M}_h$;
		\EndIf
		\State  Solve the discrete problem (\ref{trun APPLvarproblem}) on the new mesh;
		\State Compute the corresponding error estimators;
		\EndWhile 
	\end{algorithmic} 
\end{algorithm}
\subsection{Numerical Examples}
This subsection presents four numerical examples to demonstrate the robustness and effectiveness of the proposed adaptive strategy. In the first example, a scattering problem with an analytical solution is used to verify the accuracy of the method. 
The second and third examples are constructed such that the solutions have corner singularities. The last example focuses on a high-frequency case.

\vspace{12pt}
\textbf{Example~1.} In this example, the fluid-solid interface $\Gamma$ is considered as a straight line with $x_2=0$. The parameters are chosen as $\omega=\pi$, $\rho=\rho_f=1$, $\mu=\lambda=1$ and $\theta=\pi / 6$. Take $h_1=1$, $h_2=-1$ and $\Lambda=1$. According to \cite{Hu16}, the analytical solutions can be written as
\begin{align*}
	&p^{\mathrm{sc}}(\boldsymbol{x})  =q_1 e^{\mathrm{i}\left(\alpha x_1+\beta_0 x_2\right)}, \quad \boldsymbol{x} \in \Omega_+, \\
	&\boldsymbol{u}(\boldsymbol{x})  =q_2\left[\begin{array}{c}
		\alpha \\
		-\beta_0^{(1)}
	\end{array}\right] e^{\mathrm{i}\left(\alpha x_1-\beta_0^{(1)} x_2\right)}+q_3\left[\begin{array}{c}
		\beta_0^{(2)} \\
		\alpha
	\end{array}\right] e^{\mathrm{i}\left(\alpha x_1-\beta_0^{(2)} x_2\right)}, \quad \boldsymbol{x} \in \Omega_-,
\end{align*}
where the coefficients $q_j \in \mathbb{C}$ $(j=1,2,3)$ can be obtained by solving the following linear equations
\begin{align*}
	\left[\begin{array}{ccc}
		\mathrm{i} \beta_0 & \omega^2 \rho_f \beta_0^{(1)} & - \omega^2\rho_f \alpha \\
		0 & \mathrm{i} 2  \mu \alpha \beta_0^{(1)} & \mathrm{i}  2 \mu\left(\beta_0^{(2)}\right)^2-\mathrm{i} \mu \kappa_2^2 \\
		1 & \mathrm{i}2 \mu\left(\beta_0^{(1)}\right)^2+\mathrm{i} \lambda \kappa_1^2 & - \mathrm{i} 2 \mu \alpha \beta_0^{(2)}
	\end{array}\right]\left[\begin{array}{l}
		q_1 \\
		q_2 \\
		q_3
	\end{array}\right]=\left[\begin{array}{c}
		\mathrm{i} \beta_0 \\
		0 \\
		-1
	\end{array}\right].
\end{align*}
Fig. \ref{Exam1_solution k=1} presents the real parts of the exact and numerical solutions with $\kappa=1$. Obviously, the numerical solutions agree well with the exact solutions, which demonstrates the effectiveness of our PML-based adaptive FEM algorithm. Let $\boldsymbol{e}_h=\|\boldsymbol{U}-\hat{\boldsymbol{U}}_h\|_{\mathscr{H}^1_\mathrm{qp}(\Omega)}$ be the a priori error. In Fig. \ref{Exam1_error}, we show the $\log \boldsymbol{e}_h-\log \mathrm{DOF}$ and $\log \epsilon_\mathrm{F}-\log \mathrm{DOF}$ curves by utilizing the adaptive DtN-FEM method and the adaptive PML-FEM algorithm, respectively, where $\mathrm{DOF}$ denotes the number of nodes. Obviously, the convergence rates of the a priori and a posteriori errors are both $\mathcal{O}\left(\mathrm{DOF}^{-1 / 2}\right)$. It can be seen from Fig. \ref{Exam1_error} that the error of the DtN-FEM method is smaller than the PML-FEM method at the same degrees of freedom. However, the PML method only needs to impose the Dirichlet boundary conditions, while the DtN method must handle the nonlocal boundary conditions. Thus, the PML method is much simpler than the DtN method in numerical implementation.
\begin{figure}[!htb]
	\centering
	\subfigure[$\Re p^\mathrm{sc}$]{\includegraphics[width=0.32\textwidth]{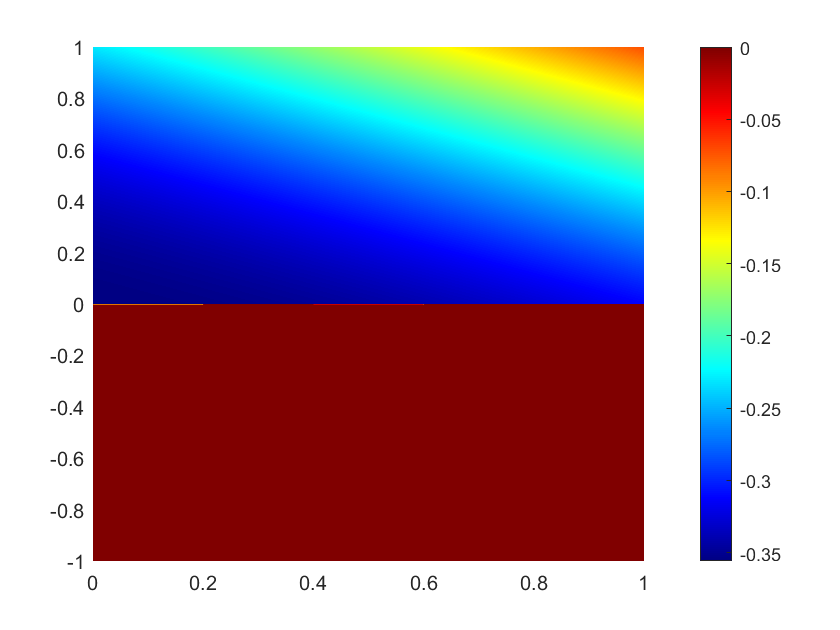}} 
	\subfigure[$\Re u_1$]{\includegraphics[width=0.32\textwidth]{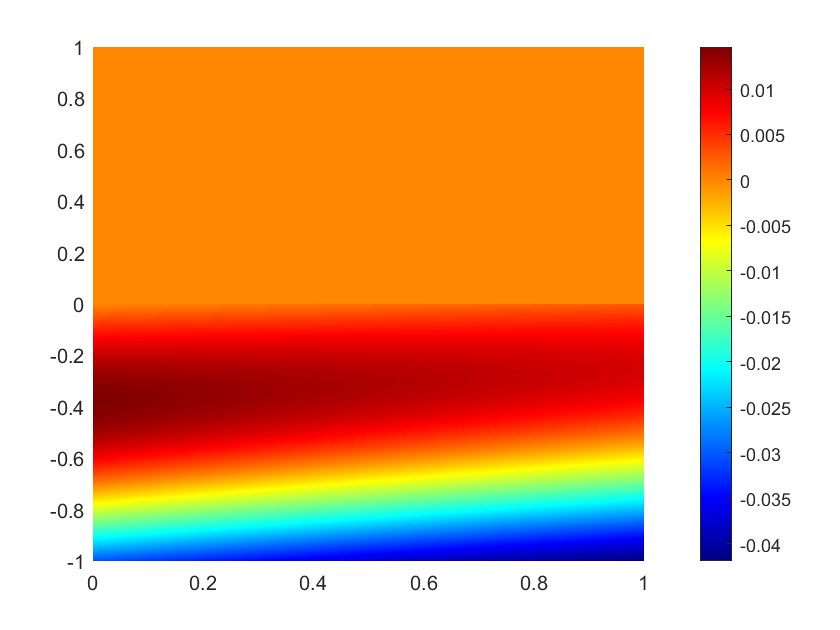}} 
	\subfigure[$\Re u_2$]{\includegraphics[width=0.32\textwidth]{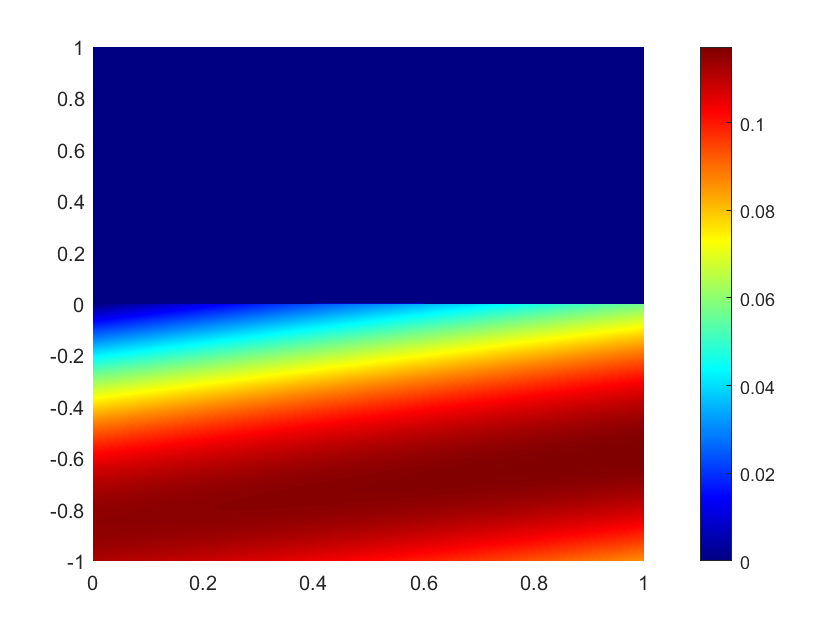}} \\
	\subfigure[$\Re \hat{p}^\mathrm{sc}_h$]{\includegraphics[width=0.32\textwidth]{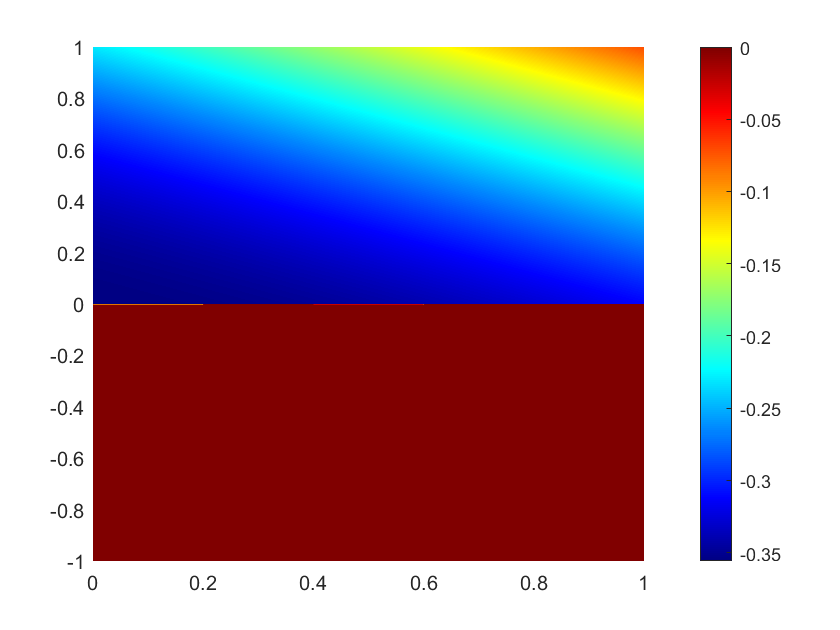}} 
	\subfigure[$\Re \hat{u}_{h,1}$]{\includegraphics[width=0.32\textwidth]{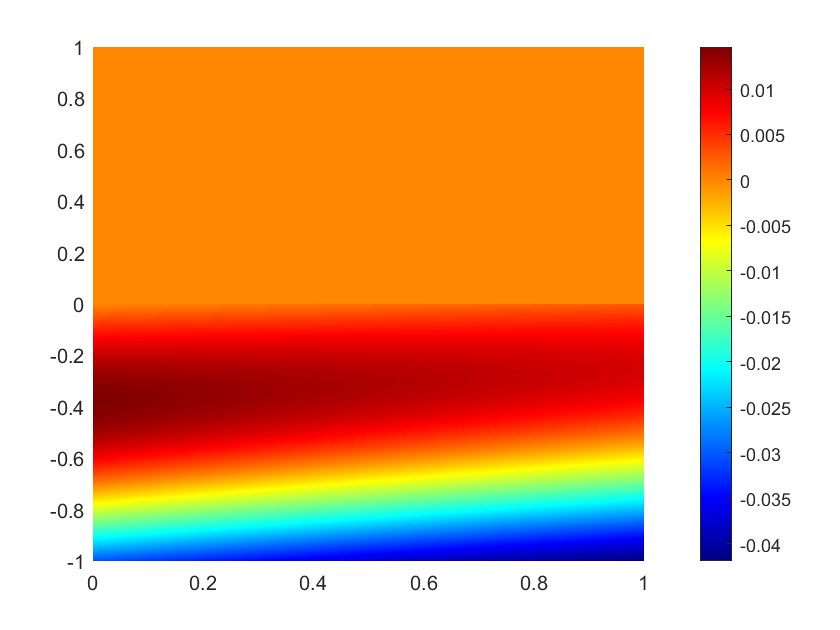}} 
	\subfigure[$\Re \hat{u}_{h,2}$]{\includegraphics[width=0.32\textwidth]{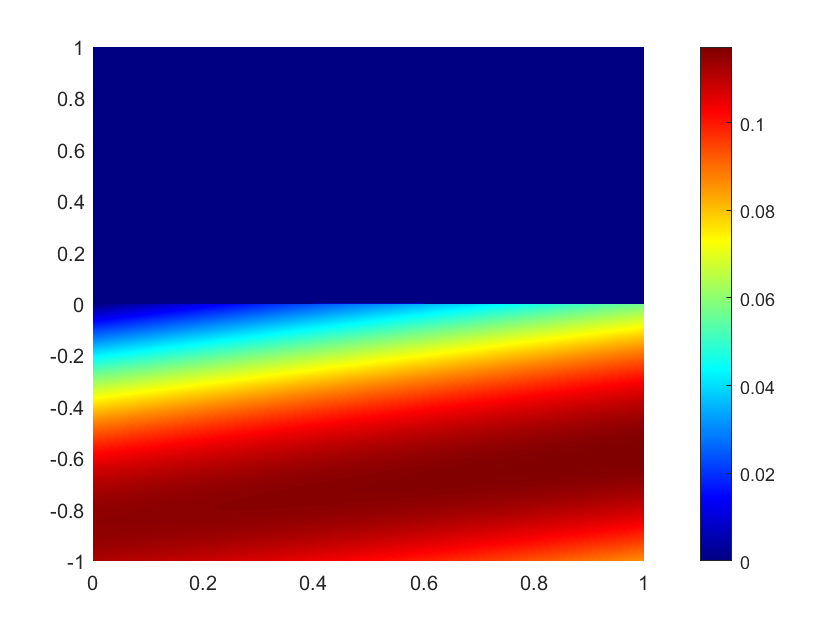}}\\
	\caption{ Example 1. Real parts of the exact solutions (a)-(c) and numerical solutions (d)-(f) with $\kappa=1$.}
	\label{Exam1_solution k=1}
\end{figure}
\begin{figure}[!htb]
	\centering 
	\subfigure[A priori error]{\includegraphics[width=0.49\textwidth]{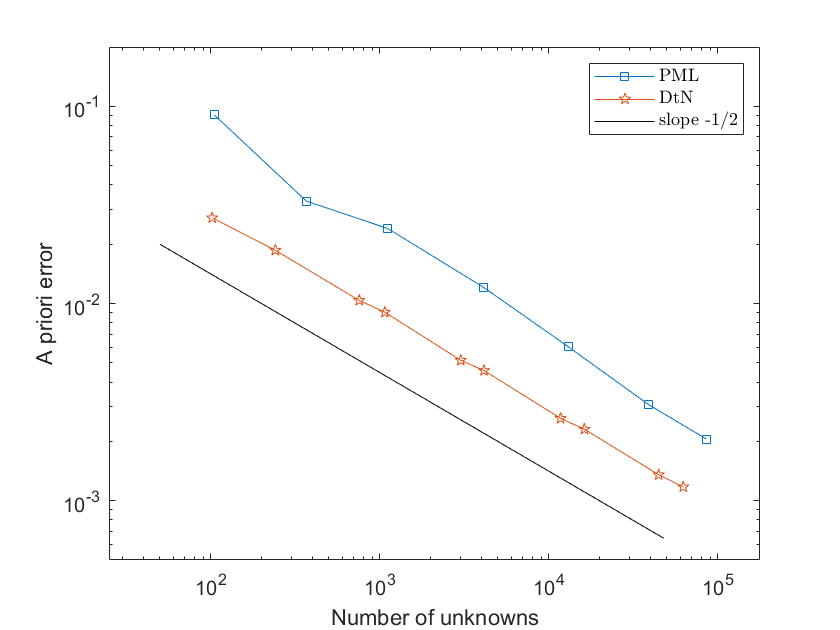}} 
	\subfigure[A posteriori error]{\includegraphics[width=0.49\textwidth]{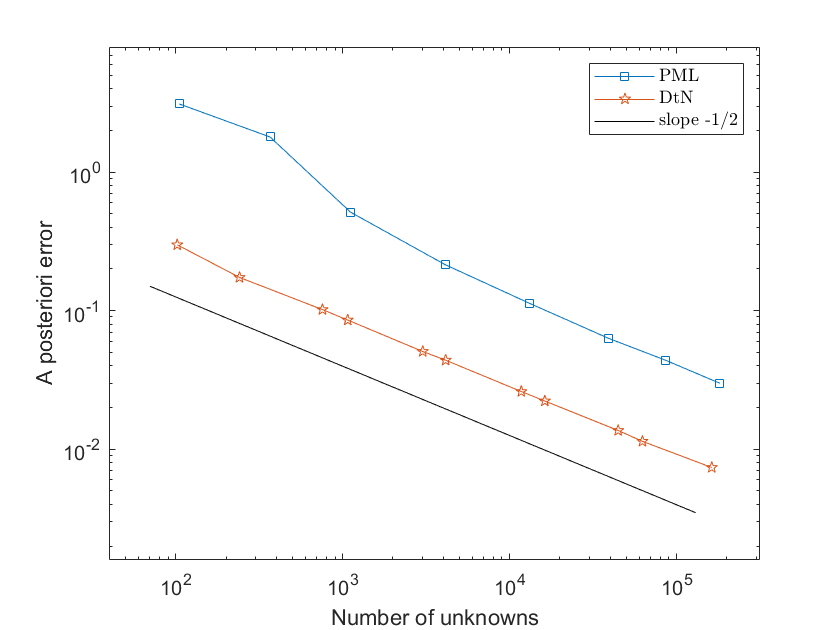}} 
	\caption{Example 1. Log-log curves of the error estimates versus $\mathrm{DOF}$ with $\delta = 1,2,3$: (a) A priori error; (b) A posteriori error.}
	\label{Exam1_error}
\end{figure}
\vspace{12pt}

\textbf{Example~2.} The grating surface is selected to have a sharp angle in this example. Let $\omega=2\pi, \rho=\rho_f=\lambda=1, \mu=2, \theta=\pi / 6$, and choose $h_1=1, h_2=-1$. We plot the numerical solutions in Fig. \ref{Exam2_solution k=1} with $\kappa=1$ and $\delta=3$. Fig.  \ref{Exam2_mesh k=1} shows the initial mesh and the corresponding adaptive mesh around the corner points after five refinement iterations, where the adaptive mesh consists of 887 nodes and 1628 elements. It is clear to show that the algorithm does capture the solution feature and adaptively refines the mesh around the corners, where solution displays singularity. The a posteriori errors and the corresponding $\mathrm{DOF}$ for the uniform and adaptive refinements with $\kappa=1$ and $\delta=3$ are presented in Table \ref{Exam2_table}. Obviously, the adaptive refinement can achieve the superior accuracy with fewer $\mathrm{DOF}$ compared with the uniform refinement, which illustrates the advantage of the adaptive strategy. Fig. \ref{Exam2_error} displays the log-log curves of the a posteriori error estimates against $\mathrm{DOF}$ with different wavenumbers $\kappa$ and PML thicknesses $\delta$, which show that the a posteriori errors are quasi-optimal.  
\begin{figure}[!htb]
	\centering 
	\subfigure[$\Im \hat{p}^\mathrm{sc}_h$]{\includegraphics[width=0.22\textwidth]{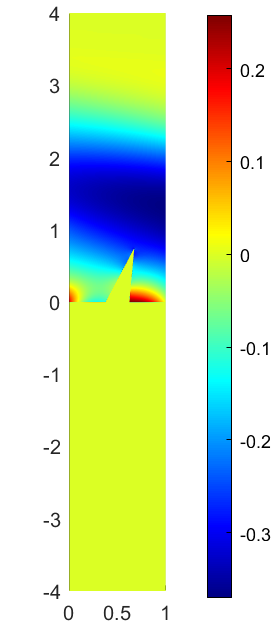}} 
	\hspace{7em} 
	\subfigure[$\Im \hat{u}_{h,1}$]{\includegraphics[width=0.22\textwidth]{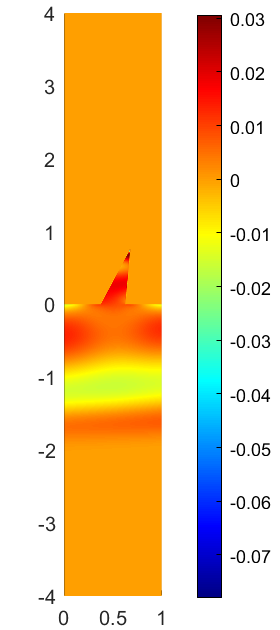}}
	\hspace{7em} 
	\subfigure[$\Im \hat{u}_{h,2}$]{\includegraphics[width=0.22\textwidth]{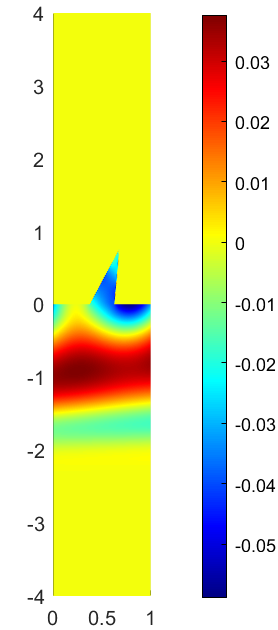}}
	\caption{Example 2. Imaginary parts of the numerical solutions with $\kappa=1$ and $\delta = 3$.}
	\label{Exam2_solution k=1}
\end{figure}
\begin{figure}[!htb]
	\centering 
	\subfigure[Initial mesh]{\includegraphics[width=0.22\textwidth]{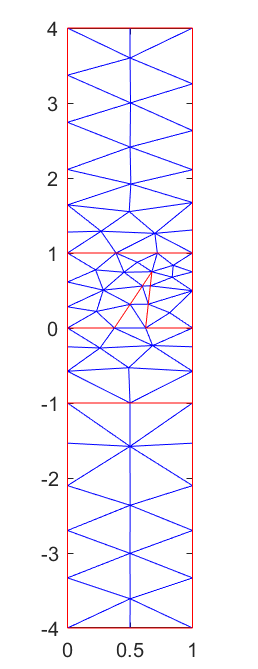}}
	\hspace{7em}
	\subfigure[Adaptive mesh]{\includegraphics[width=0.22\textwidth]{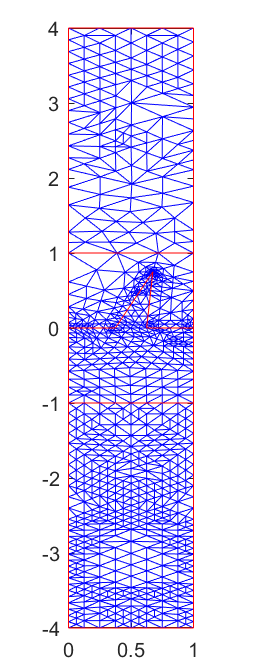}}
	\caption{Example 2. The finite element mesh: (a) Initial mesh with 70 nodes and 100 elements; (b)Adaptive mesh with 887 nodes and 1628 elements.}
	\label{Exam2_mesh k=1}
\end{figure}
\begin{table*}[!htb]
	\centering
	\caption{Example 2. Comparison of the uniform and adaptive refinements in terms of $\mathrm{DOF}$ and the a posteriori error $\epsilon_\mathrm{F}$.}
	\begin{tabular*}{\textwidth}{@{\extracolsep\fill}llllllll@{\extracolsep\fill}}
		\toprule
		\multicolumn{2}{c}{Uniform refinement} & & &      \multicolumn{2}{c}{Adaptive refinement} \\
		\cmidrule(r){1-5} \cmidrule(l){5-8}
		$\mathrm{DOF}$ & \hspace{3em} $\epsilon_\mathrm{F}$  & & &  $\mathrm{DOF}$ & \hspace{3em}$\epsilon_\mathrm{F}$  \\
		\midrule
		70 & \hspace{3em}5.1176 &      & &   70 & \hspace{3em}5.1176 &  \\
		3353 & \hspace{3em}0.5986 &    & &   2842 & \hspace{3em}0.4073 &  \\
		13105 & \hspace{3em}0.2543 &   & &   9786 & \hspace{3em}0.2076 &  \\
		51809 & \hspace{3em}0.1495 &   & &   19700 & \hspace{3em}0.1469 &  \\
		206017 & \hspace{3em}0.1014 &  & &   48264 & \hspace{3em}0.0962 &  \\
		\bottomrule
		\label{Exam2_table}
	\end{tabular*}
\end{table*}
\begin{figure}[!htb]
	\centering 
	\subfigure[A posteriori errors with different $\kappa$ ]{\includegraphics[width=0.49\textwidth]{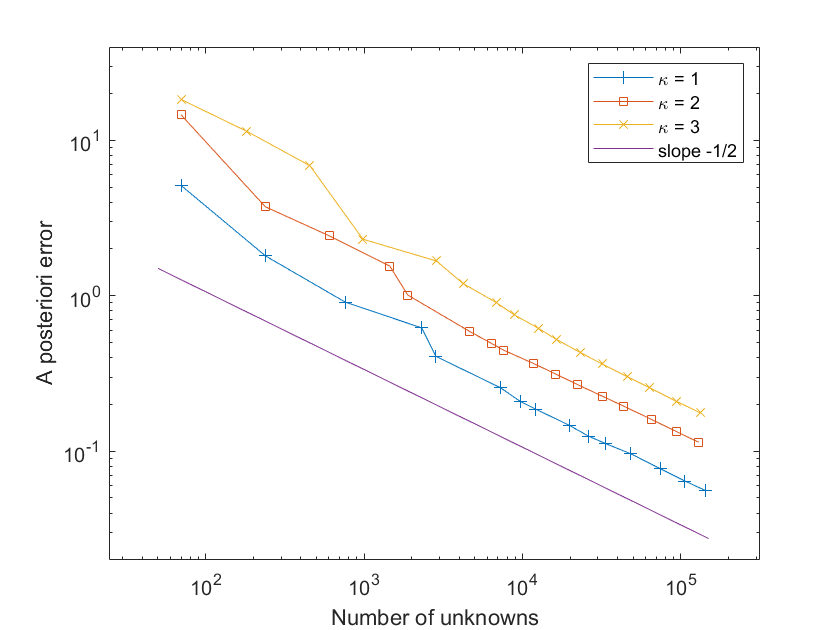}}
	\subfigure[A posteriori errors with different $\delta$ ]{\includegraphics[width=0.49\textwidth]{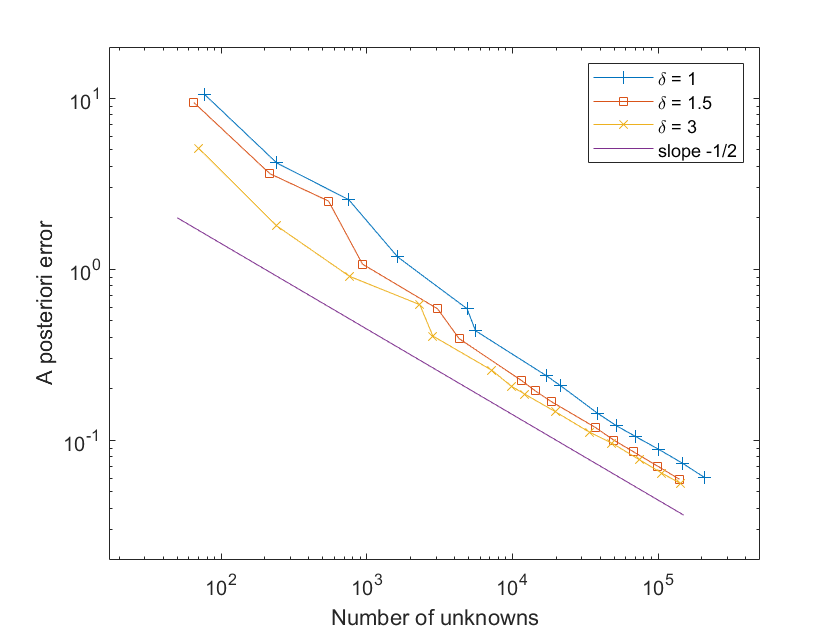}}
	\caption{Example 2. Log-log curves of the a posteriori error estimates versus $\mathrm{DOF}$: (a) For wavenumbers $\kappa=1,2,3$; (b) For PML thicknesses $\delta=1,1.5,3$.}
	\label{Exam2_error}
\end{figure}
\vspace{12pt}

\textbf{Example~3.} In this example, the grating surface is a trapezoidal structure. We choose $h_1=2$, $h_2=-2$ and $\Lambda=4$. The rest of the parameters are the same as those in Example 2. We set $\kappa =1$, $\delta=4$. The real parts of the numerical solutions as well as the corresponding adaptive mesh are displayed in Figs. \ref{Exam3_solution k=1} and \ref{Exam3_mesh k=1}, respectively. A numerical comparison between uniform and adaptive refinements is provided in Table \ref{Exam3_table}. Once again, the algorithm shows the ability
to capture the singularity of the solution and perform local mesh refinements. In Fig. \ref{Exam3_error}, the curves of $\log \epsilon_\mathrm{F}-\log \mathrm{DoF}$ are plotted for different $\kappa$ and $\delta$ to demonstrate the quasi-optimal convergence rates of our proposed method.
\begin{figure}[!htb]
	\centering
	\subfigure[$\Re \hat{p}^\mathrm{sc}_h$]{\includegraphics[width=0.26\textwidth]{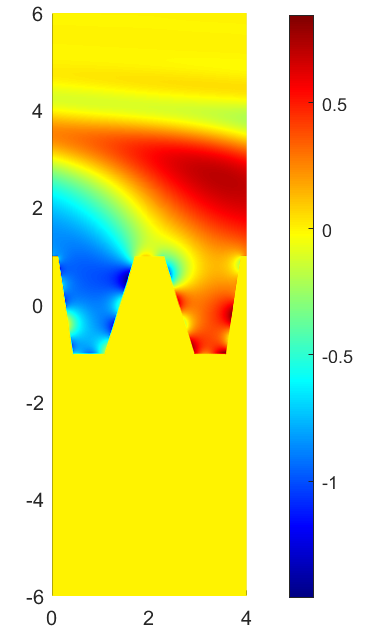}} 
	\hspace{4em} 
	\subfigure[$\Re \hat{u}_{h,1}$]{\includegraphics[width=0.26\textwidth]{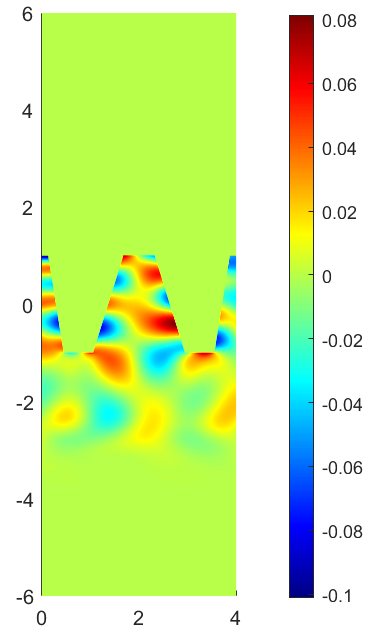}}
	\hspace{4em} 
	\subfigure[$\Re \hat{u}_{h,2}$]{\includegraphics[width=0.26\textwidth]{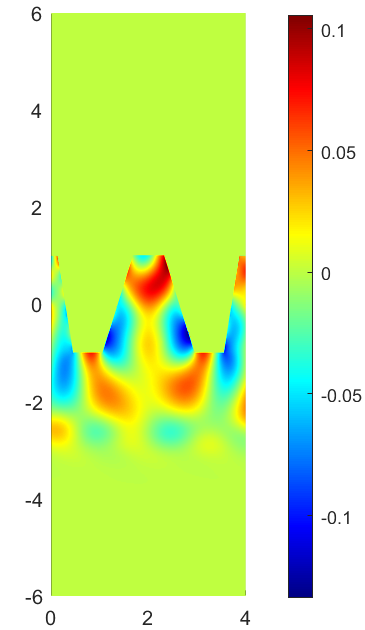}}
	\caption{Example 3. Real parts of the numerical solutions with $\kappa=1$ and $\delta = 4$.}
	\label{Exam3_solution k=1}
\end{figure}
\begin{figure}[!htb]
	\centering 
	\subfigure[Initial mesh]{\includegraphics[width=0.32\textwidth]{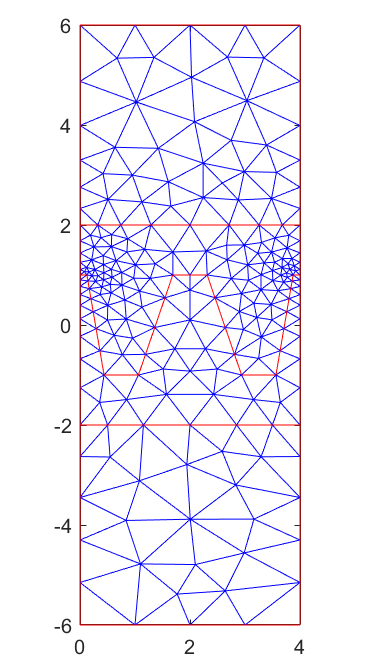}}
	\hspace{6em}
	\subfigure[Adaptive mesh]{\includegraphics[width=0.32\textwidth]{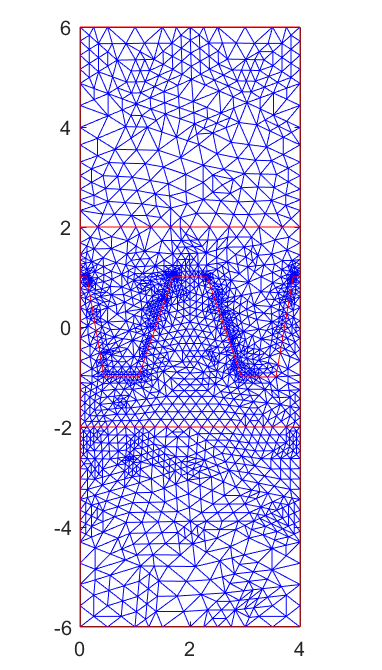}}
	\caption{Example 3. The finite element mesh: (a) Initial mesh with 232 nodes and 404 elements; (b) Adaptive mesh with 1963 nodes and 3793 elements.}
	\label{Exam3_mesh k=1}
\end{figure}
\begin{table*}[!htb]
	\centering
	\caption{Example 3. Comparison of the uniform and adaptive refinements in terms of $\mathrm{DOF}$ and the a posteriori error $\epsilon_\mathrm{F}$.}
	\begin{tabular*}{\textwidth}{@{\extracolsep\fill}llllllll@{\extracolsep\fill}}
		\toprule
		\multicolumn{2}{c}{Uniform refinement} & & &      \multicolumn{2}{c}{Adaptive refinement} \\
		\cmidrule(r){1-5} \cmidrule(l){5-8}
		$\mathrm{DOF}$ & \hspace{3em} $\epsilon_\mathrm{F}$  & & &  $\mathrm{DOF}$ & \hspace{3em}$\epsilon_\mathrm{F}$  \\
		\midrule
		232 & \hspace{2.1em}12.1684 &      & &   232 & \hspace{2.1em}12.1684 &  \\
		867 & \hspace{2.1em}8.1307 &     & &   610 & \hspace{2.1em}5.5736 &  \\
		13161 & \hspace{2.1em}2.4135 &   & &  6455 & \hspace{2.1em}2.3914 &  \\
		52177 & \hspace{2.1em}1.2630 &   & &   27369 & \hspace{2.1em}1.2395 &  \\
		207777 & \hspace{2.1em}0.6804 &    & &   119932 & \hspace{2.1em}0.5919 &  \\
		\bottomrule
		\label{Exam3_table}
	\end{tabular*}
\end{table*}
\begin{figure}[!htb]
	\centering 
	\subfigure[A posteriori errors with different $\kappa$ ]{\includegraphics[width=0.49\textwidth]{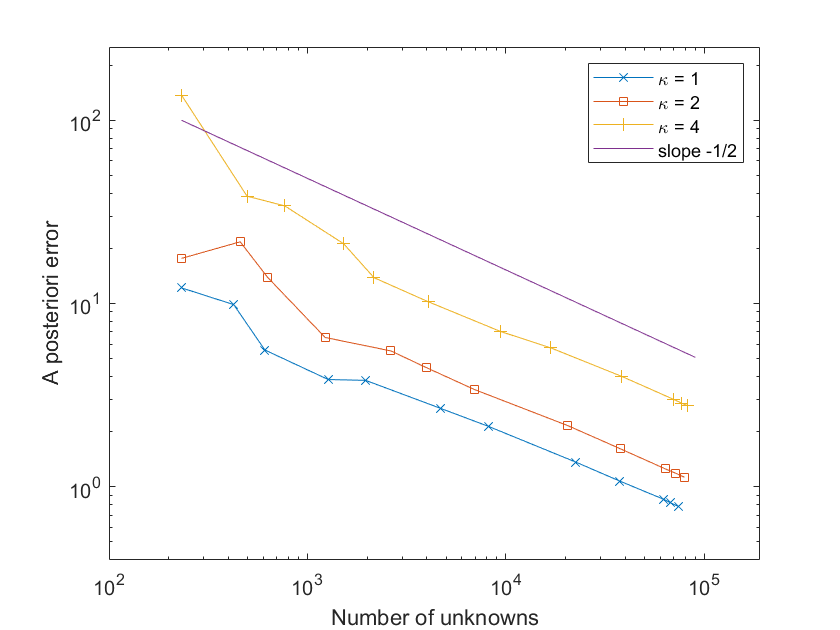}} 
	\subfigure[A posteriori errors with different $\delta$]{\includegraphics[width=0.49\textwidth]{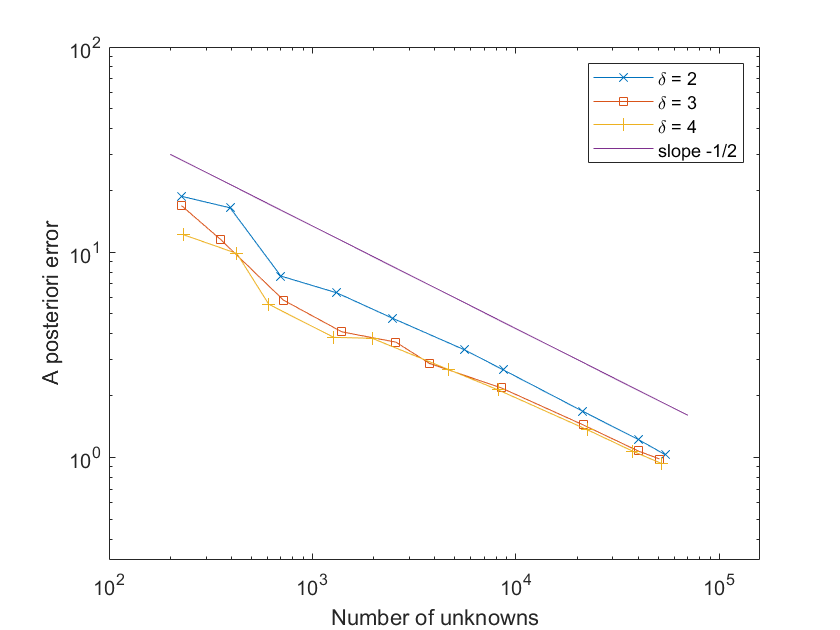}} 
	\caption{Example 3. Log-log curves of the a posteriori error estimates versus $\mathrm{DOF}$: (a) For wavenumbers $\kappa=1,2,4$; (b) For PML thicknesses $\delta=2,3,4$.}
	\label{Exam3_error}
\end{figure}
\vspace{12pt}

\textbf{Example~4.} This example focuses on a high-frequency problem, where the grating surface has two sharp angles. We take the period $\Lambda=2$. The remaining parameters are the same as those in Example 3.
The real parts of numerical solutions with $\kappa=20$ are presented in Fig. \ref{Exam4_solution k=20}. Fig. \ref{Exam4_error} (a) demonstrates the quasi-optimality of the a posteriori error with different $\delta$. Fig. \ref{Exam4_error} (b) displays the a posteriori errors against $\mathrm{DOF}$ with fixed $\delta=4$ and different $\kappa=10,20,30$, which demonstrates that our adaptive PML-FEM algorithm is robust even for high-frequency problems. In addition, numerical results in Table \ref{Exam4_k20_table} further confirm the superiority of the adaptive strategy.  
\begin{figure}[!htb]
	\centering 
	\subfigure[$\Re \hat{p}^\mathrm{sc}_h$]{\includegraphics[width=0.265\textwidth]{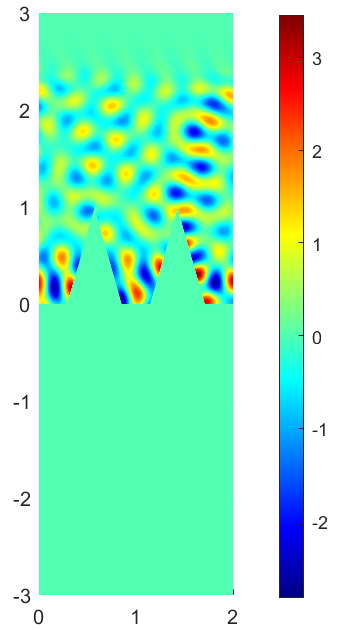}} 
	\hspace{4em} 
	\subfigure[$\Re \hat{u}_{h,1}$]{\includegraphics[width=0.265\textwidth]{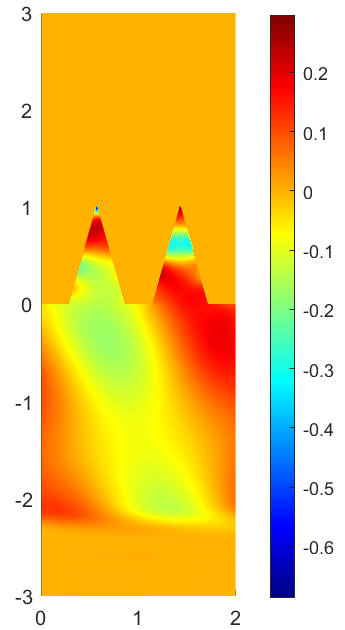}}
	\hspace{4em} 
	\subfigure[$\Re \hat{u}_{h,2}$]{\includegraphics[width=0.265\textwidth]{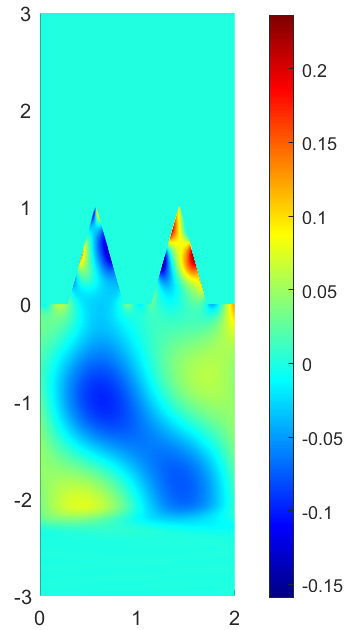}}
	\caption{Example 4. Real parts of the numerical solutions with $\kappa=20$ and $\delta = 1$.}
	\label{Exam4_solution k=20}
\end{figure}
\begin{figure}[h]
	\centering 
	\subfigure[A posteriori errors with different $\delta$]{\includegraphics[width=0.49\textwidth]{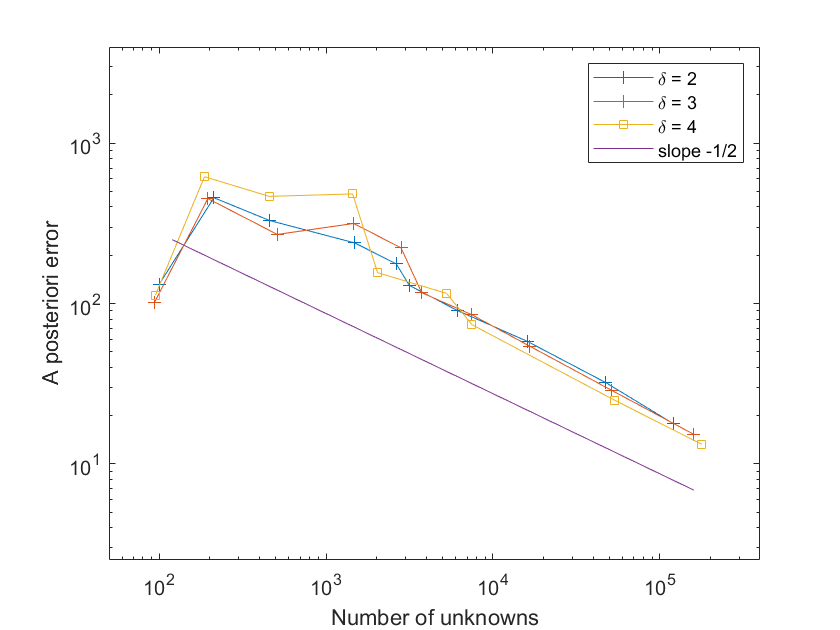}} 
	\subfigure[A posteriori errors with different $\kappa$ ]{\includegraphics[width=0.49\textwidth]{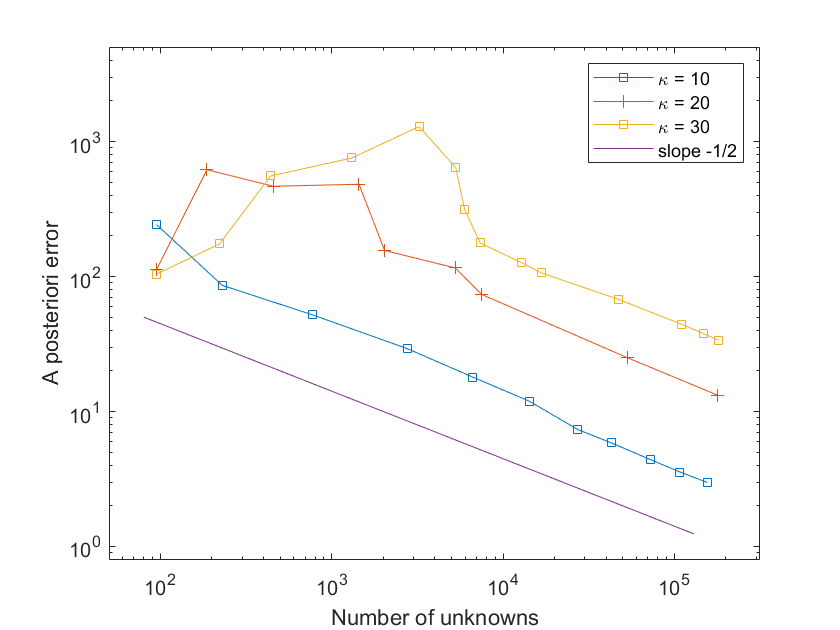}} 
	\caption{Example 4. Log-log curves of the a posteriori error estimates versus $\mathrm{DOF}$: (a) For PML thicknesses $\delta=2,3,4$; (b) For wavenumbers $\kappa=10,20,30$.}
	\label{Exam4_error}
\end{figure}

\begin{table*}[!htb]
	\centering
	\caption{Example 4. Comparison of the uniform and adaptive refinements in terms of $\mathrm{DOF}$ and the a posteriori error $\epsilon_\mathrm{F}$.}
	\begin{tabular*}{\textwidth}{@{\extracolsep\fill}llllllll@{\extracolsep\fill}}
		\toprule
		\multicolumn{2}{c}{Uniform mesh} & & &      \multicolumn{2}{c}{Adaptive mesh} \\
		\cmidrule(r){1-5} \cmidrule(l){5-8}
		$\mathrm{DOF}$ & \hspace{3em} $\epsilon_\mathrm{F}$  & & &  $\mathrm{DOF}$ & \hspace{3em}$\epsilon_\mathrm{F}$  \\
		\midrule
		95 & \hspace{1.8em}112.6802 &      & &   95 & \hspace{1.8em}112.6802 &  \\
		4897 & \hspace{1.8em}136.5494 &     & &   2301 & \hspace{1.8em}98.6562 &  \\
		19265 & \hspace{1.8em}105.3166 &   & &   4342 & \hspace{1.8em}78.3530 &  \\
		76417 & \hspace{1.8em}44.0865 &   & &   21979 & \hspace{1.8em}41.1862 &  \\
		304385 & \hspace{1.8em}21.6798 &  & &  149196 & \hspace{1.8em}14.7741 &  \\
		\bottomrule
		\label{Exam4_k20_table}
	\end{tabular*}
\end{table*}

\section{Conclusion}\label{SECTION6}
We develop a PML-based adaptive FEM for solving the acoustic-elastic interaction problem in periodic structures. 
By introducing two PML-equivalent transparent boundary conditions, we prove the unique solvability of the truncated PML variational formulation and show that the PML truncation error decays exponentially as the PML parameters increase. We further give a residual-type a posteriori error estimate and develop an adaptive FEM algorithm to address singularities of the solution caused by the non-smooth fluid-solid interface. Numerical experiments are presented to validate the accuracy and robustness of the proposed method. Future work will extend the analysis of PML-AFEM to the thermoelastic scattering problem.



\section*{Acknowledgment}
This work of J.L. was partially supported by the National Natural Science Foundation of China (Grant No. 12271209).

\section*{Data availability}
All data generated or analysed during this study are included in this article.

\section*{Declaration of competing interest}
The authors declare that they have no known competing financial interests or personal relationships that could have appeared to influence the work reported in this paper.

\end{document}